\documentclass[11pt,oneside,reqno]{amsart}
\usepackage{mathpazo} 
\normalfont
\usepackage[T1]{fontenc}

\usepackage[utf8]{inputenc}
\usepackage{amsmath}
\usepackage{amsfonts}
\usepackage{amssymb}
\usepackage{amsthm}
\usepackage{tikz}
\usepackage{tikz-cd}
\usepackage[all]{xy}
\usepackage[margin=1.2in]{geometry}
\usepackage{amscd}
\usepackage[shortlabels]{enumitem}

\DeclareMathOperator{\Hom}{Hom}

\newcommand{\angles}[1]{\left\langle #1 \right\rangle}

\newcommand{\supp}[1]{\mathrm{supp} #1 }

\input xy
\xyoption{all}
\theoremstyle{definition}
\newtheorem{mydef}{\textbf{Definition}}[section]
\newtheorem{myeg}[mydef]{\textbf{Example}}

\newtheorem{rmk}[mydef]{\textbf{Remark}}

\newtheorem{notation}[mydef]{\textbf{Notation}}

\theoremstyle{plain}
\newtheorem{mythm}[mydef]{\textbf{Theorem}}

\newtheorem*{nothma}{\textbf{Theorem A}}
\newtheorem*{nothmb}{\textbf{Theorem B}}
\newtheorem*{nothmc}{\textbf{Theorem C}}
\newtheorem*{nothmd}{\textbf{Theorem D}}

\newtheorem{lem}[mydef]{\textbf{Lemma}}
\newtheorem{pro}[mydef]{\textbf{Proposition}}

\newcommand{\R}{\mathbb{R}}

\begin{document}

\title{Closure Operators and Geometric Modules of Valuated Matroids}

\author{Jaiung Jun}
\address{Department of Mathematics, State University of New York at New Paltz, NY, USA}
\email{junj@newpaltz.edu}

\author{Jeffrey Tolliver}
\address{}
\email{jeff.tolli@gmail.com}
\makeatletter
\@namedef{subjclassname@2020}{%
	\textup{2020} Mathematics Subject Classification}
\makeatother

\subjclass[2020]{12K10, 14T10, 05B35, 06C10}
\keywords{Matroid, closure operator, valuated matroid, geometric lattice, flat, tropical linear space}
\thanks{}

\begin{abstract}
We introduce a closure operator for valuated matroids and prove that it yields a cryptomorphic definition of valuated matroids. As an application, we introduce a class of modules over the tropical semifield (geometric modules) and prove that there is one-to-one correspondence between projective equivalence classes of simple valuated matroids and isomorphism classes of finitely generated geometric modules. We further illustrate how this correspondence can be lifted to simple infinite valuated matroids and geometric modules. 
\end{abstract}

\maketitle

\section{Introduction}

Matroid theory has developed into a major area of combinatorics over the past century since its introduction by Whitney \cite{whitney1935abstract}  and Nakasawa \cite{nakasawa1935axiomatik}. Since work of Edmonds, it has played an important role in combinatorial optimization.

Some early work on valuated matroids was motivated by an area of combinatorial optimization known as discrete convex analysis. For instance, see \cite{murota2001circuit}. 
However a bigger source of interest these days comes from the observation that the tropicalizations of vector spaces are valuated matroids, and hence these serve as the most fundamental class of examples in tropical geometry \cite{speyer2008tropical}. In fact, in tropical geometry one may see valuated matroids as points in Grassmannian defined over the tropical semifield $\mathbb{R}_{\max}$ (Example \ref{example: tropical semimifield}). This perspective has a greater generalization, now known as matroids with coefficients, initiated by Dress and Wenzel \cite{dress1986duality}, and further developed by Baker,  Bowler, and Lorscheid \cite{baker2019matroids}, \cite{baker2021moduli}. This approach to matroid theory further found its applications, for instance, by Baker, Huh, Kummer, and Lorscheid \cite{baker2025lorentzian}.

Like matroids, valuated matroids have many cryptomorphic descriptions, such as via Grassmann-Pl\"ucker functions, circuits, cocircuits, vectors and covectors.  However, several of the classical ways of defining a matroid do not yet have known analogues for valuated matroids. This paper is focused on describing some additional ways to define a valuated matroid, namely closure operators and geometric modules over the tropical semifield. Our results still hold if one replaces $\mathbb{R}$ with any totally ordered abelian group, but we restrict ourselves to  $\mathbb{R}$. 

There is a known cryptomorphic description of oriented matroids in terms of closure operators \cite{buchi1988large}.  In the linear case, this amounts to mapping subset $S\subseteq E$ to the set of elements of $E$ which are non-negative linear combinations of elements of $S$. That is instead of taking all linear combinations, we must restrict the coefficients in a way that depends on the ordering of the field.  In the valuated case, the restriction we place on coefficients should involve the valuation, so it is natural to expect the closure of a linear valuated matroid to correspond to the $\mathcal{O}_K$-linear span.  In fact the operation of taking the $\mathcal{O}_K$-linear span has appeared in previous work on valuated matroids and on tropical geometry, e.g. \cite[Proposition 5.1]{hirai2019uniform} or \cite[Lemma 7.2]{JMT20}.

To define a closure operator for valuated matroids, of course we should not expect a closure operator on the ground set $E$ of a valuated matroid to be able to fully capture the structure.  Instead we define a closure operator on the related set
\[
\hat{M} = \mathbb{R}_{\max} \times E / \{(-\infty, x\}\mid x\in E\}.
\]
We use $\leq$ to denote the partial order on $\hat{M}$ defined as follows:
\[
ax\leq by \iff \text{either both sides are $\textbf{0}$ or $x=y$ and $a\leq b$,}
\]
where $ax:=(a,x)$ and $\textbf{0}:=(-\infty,x)$ for any $x \in E$. 

We first characterize closure operators on $\hat{M}$ which precisely yield valuated matroids.

\begin{nothma}[Proposition \ref{proposition: valuated matroid closure cryptomorphism}]Let $E$ be a finite set and $\hat{M}$ be as above.  There is a one-to-one correspondence between valuated matroids on $E$ and algebraic closure operators on $\hat{M}$ satisfying the following:
\begin{enumerate}
    \item $\textbf{0}\in\mathrm{cl}(\emptyset)$.
    \item For any $S\subseteq \hat{M}$ and $a\in \mathbb{R}_{\max}^\times$, $\mathrm{cl}(aS) = a \mathrm{cl}(S)$. 
    \item For any $S\subseteq \hat{M}$, $\mathrm{cl}(S)$ is downward closed.
    \item Let $S\subseteq\hat{M}$ and let $x,y\in E$ and $a\in \mathbb{R}_{\max}$.  Suppose there is some $b\in \mathbb{R}_{\max}$ such that $ax\in \mathrm{cl}(S\cup\{by\})$.  Then there is some $c\leq b$ such that $ax\in \mathrm{cl}(S\cup\{cy\})$ and $cy\in \mathrm{cl}(S \cup\{ax\})$.
    \item Let $S\subseteq\hat{M}$ and let $x\in E$ and $a, b\in \mathbb{R}_{\max}$.  Suppose $ax\in \mathrm{cl}(S \cup \{bx\})$ and that $b < a$.  Then $ax\in \mathrm{cl}(S)$.
\end{enumerate}
\end{nothma}

Let $\mathbb{B}=\{0,1\}$ be the Boolean semifield viewed in $\mathbb{R}_{\max}$ (Example \ref{example: Boolean}). One may observe that the conditions in Theorem A precisely describe the closure operator of a matroid if we replace $\mathbb{R}_{\max}$ with $\mathbb{B}$. In fact, the first three conditions and the fifth are trivial in this case. Condition (4) is equivalent to the Mac Lane-Steinitz exchange property of closure operators of matroids.

To prove Theorem $A$, we first define flats of valuated matroid and define the closure $\text{cl}(S)$ for $S \subseteq \hat{M}$ to be the intersections of all flats containing $S$ (Definition \ref{definition: flats}). We then prove that $\text{cl}$ is indeed an algebraic closure operator satisfying the conditions in Theorem $A$, and all algebraic closure operators satisfying the conditions arise in this way.

In a linear matroid (associated to a vector space over a field $K$), the closure operator sends a set $S$ to the set of elements of the matroid which lie within the $K$-linear span of $S$.  By contrast, in the valuated matroid case, where we assume that $K$ is a valued field whose valuation map is surjective, $\mathrm{cl}(S)$ corresponds to the $\mathcal{O}_K$-linear span. 

\begin{nothmb}[Proposition \ref{proposition: flats of linear valuated matroids}]
Let $V$ be a vector space over a valued field $K$ with surjective valuation and $M$ the a linear valuated matroid induced by a finite set $E\subseteq V$.  Given an $\mathcal{O}_K$-submodule $N\subseteq V$, let 
\begin{equation}
    F_N = \{ v(c) x \mid x\in E, c\in K, cx\in N\}.
\end{equation}  
Then, the following hold. 
\begin{enumerate}
    \item 
$F\subseteq \hat{M}$ is a flat if and only if $F = F_N$ for some $N$.
    \item 
If $F\subseteq \hat{M}$ is a flat, there is exactly one finitely generated $\mathcal{O}_K$-submodule $N\subseteq V$ such that $N$ is generated by elements which have the form $ax$ for $a\in\mathbb{R}_{\max}$ and $x\in E$ and such that $F = F_N$.  In particular, flats of $M$ are in one-to-one correspondence with $\mathcal{O}_K$-submodules of $V$ which are generated by elements of this form.
\end{enumerate}
\end{nothmb}

Jeffrey and Noah Giansiracusa have introduced a finitely generated $\mathbb{R}_{\max}$-module $Q_M$ associated to a valuated matroid, which is defined via a certain presentation in terms of generators and relations \cite{giansiracusa2018grassmann}, called \emph{bend relations}.\footnote{This relation reflects how the locus of a polynomial with coefficients in $\mathbb{R}_{\max}$ is defined in tropical geometry (see Section \ref{subsection: semirings}).}  For a classical matroid, this construction instead yields a finitely generated $\mathbb{B}$-module (equivalently a finite lattice).  It was proven in \cite[Theorem 1.1]{crowley2020module} that in the case of a classical matroid, $Q_M$ can be identified with the lattice of flats.  For valuated matroids, we have the following.

\begin{nothmc}[Proposition \ref{proposition: iso $Q_M$ to flats}]
Let $M$ be a valuated matroid on a finite set $E$.  Let $\mathcal{S}$ be the $\mathbb{R}_{\max}$-module of finitely generated flats.  Then $\mathcal{S}$ is canonically isomorphic to $Q_M$.
\end{nothmc}

We remark that this indirectly yields a relationship between flats and the tropical linear space $V_M$ associated to a valuated matroid $M$, since $V_M$ is the dual module of $Q_M$ in the sense that $V_M=\Hom_{\mathbb{R}_{\max}}(Q_M,\mathbb{R}_{\max})$.  However, unlike the case of $\mathbb{B}$-modules, module-theoretic duality is quite different from lattice-theoretic duality.

A related question is whether there is a valuated analogue of the geometric lattice characterization of matroids.  Since a finite lattice is the same as a finitely-generated $\mathbb{B}$-module, it is natural to look for a characterization of valuated matroids in terms of $\mathbb{R}_{\max}$-modules.  We define conditions analogous to atomicity and semimodularity for $\mathbb{R}_{\max}$-modules.  Besides these two conditions, we need an additional hypothesis (strong torsion-freeness), which is automatic for $\mathbb{B}$-modules but not for $\mathbb{R}_{\max}$-modules. 

Recall from \cite{dress1992valuated} that for two valuated matroid structures $w$ and $w'$ on $E$ with the same underlying matroid $M$, we say that $w$ and $w'$ are projectively equivalent if there exist a function $s:E \to (\mathbb{R}_{\max})^\times$ such that for any base $B$ of $M$, one has
\begin{equation}\label{eq: proj eq}
w(B) = \left(\prod_{i \in B}s(i)\right) w'(B). 
\end{equation}
One can further generalize this definition to introduce projective equivalence between two valuated matroids (Definition \ref{definition: proj eq}). We prove that the three module-theoretic conditions (atomicity, semimodularity, and strong torsion-freeness) precisely characterize simple valuated matroids up to projective equivalence as follows.  

\begin{nothmd}[Proposition \ref{proposition: $Q_M$ geometric}, Theorem \ref{theorem: geometric modules1}, Theorem \ref{theorem: geometric modules2}]
Let $M$ be a valuated matroid on a finite set $E$. Then, the following hold. 
\begin{enumerate}
    \item 
$Q_M$ is atomic, semimodular and strongly torsion-free as $\mathbb{R}_{\max}$-module.
\item 
Every finitely generated atomic semimodular and strongly torsion-free $\mathbb{R}_{\max}$-module arises in this way from a simple valuated matroid.  In addition, this valuated matroid is unique up to projective equivalence.
\end{enumerate}
\end{nothmd}

In proving Theorem D, one can observe that finiteness condition on $E$ for simple valuated matroids precisely corresponds to the finite generation condition for geometric modules. In particular, one may remove the finiteness condition for both sides, and obtain a similar correspondence. This is discussed in Appendix \ref{section: infinite val mat}. This perspective allows one to view any coordinate semiring of an affine tropical variety as a geometric $\mathbb{R}_{\max}$-module associated to an infinite valuated matroids whose underlying set consists of monomials. See, Example \ref{example: tropical var is val}.

\begin{rmk}
As this paper was nearing completion we became aware of the work of Anderson \cite{andersonthesis}, where he essentially proves a similar result to Theorem A, built on the ideas of Maclagan and Rinc\'on \cite{maclagan2016tropical}. But, we believe that our perspective is more useful to generalize Theorem A to matroids over idylls as in Baker and Lorscheid \cite{baker2021moduli}. For instance, see Remark \ref{remark: idyll} to see how downward closed condition may be rephrased. On the contrary, Anderson's approach does not seem to be immediately generalizable to idylls because it is a closure operator on the poset of maps $E \to \mathbb{R}_{\max}$, and the fact that this is even a poset is very specific to the tropical case. With our definition, it turns out that the join semilattice of finitely generated flats also carries a structure as a module over the idyll. 
\end{rmk}

\bigskip

\textbf{Acknowledgment} J.~Jun was partially supported by NSF LEAPS-MPS (DMS-2532394) and AMS-Simons Research Enhancement Grant for Primarily Undergraduate Institution (PUI) Faculty during the writing of this paper.

\bigskip

\section{Preliminaries}\label{section: preliminaries}

\subsection{Semirings and Algebraic closure operators} \label{subsection: semirings}

\begin{mydef} $ $
\begin{enumerate}
\item 
A \emph{semiring} is a nonempty set $S$ with two binary operations $+$ and $\cdot$ such that $(S,+,0)$ and $(S,\cdot,1)$ are commutative monoids such that 
\[
a\cdot 0 =0, \quad a\cdot(b+c)=a\cdot b +a\cdot c, \quad\forall~a,b,c \in S.
\] 
\item
A semiring $S$ is said to be a \emph{semifield} if $(S\backslash \{0\},\cdot)$ is a group.
\item 
A semiring $S$ is said to be (additively) \emph{idempotent} if $a+a=a$ for all $a \in S$, or equivalently $1+1=1$. 
\end{enumerate}
\end{mydef}

For a semiring $S$, by an $S$-module, we mean a monoid $M$ with action of $S$ satisfying the same axioms as modules over rings. Note that an idempotent semiring is equipped with a partial order: $x \leq y \iff x+y=y$. The following are two main examples.

\begin{myeg}[Tropical semifield]\label{example: tropical semimifield}
Let $\mathbb{R}_{\text{max}}=\mathbb{R} \cup \{-\infty\}$ and the multiplication of $\mathbb{R}_{\text{max}}$ is the usual addition of real numbers and the addition is given by taking the maximum, i.e.,  for $a, b \in \mathbb{R}_{\text{max}}$,
\[
a+b:=\max \{a,b\}.
\]
\end{myeg}

\begin{myeg}[Boolean semifield]\label{example: Boolean}
Let $\mathbb{B}=\{0,1\}$. Addition is given as follows:
\[
1+0=1, \quad 0+0=0, \quad 1+1=1. 
\]
Multiplication is as follows:
\[
1\cdot 1= 1, \quad 1\cdot 0 =0, \quad 0 \cdot 0=0. 
\]
\end{myeg}

Now, we recall two $\mathbb{R}_{\max}$-modules $Q_M$ and $V_M$ associated to a valuated matroid $M$. For $f=\sum_{u \in U \subseteq \mathbb{N}^n} a_ux^u \in \mathbb{R}_{\max}[x_1,\dots,x_n]$, the bend relations of $f$ are the relations of the following form:
\begin{equation}\label{eq: bend}
\left\{ f \sim \sum_{v \neq u}a_vx^v                 \right\}_{u \in U}
\end{equation}
Let $\mathcal{B}(f)$ be the congruence relation generated by \eqref{eq: bend}.\footnote{By a congruence relation we mean an equivalence relation which is compatible with addition and multiplication.} 

For a given valuated matroid $M$ on a finite set $E$, one obtains a set $I$ of linear equations in $\mathbb{R}_{\max}\langle E \rangle$. Then, with 
\[
\mathcal{B}(I)=\angles{ \mathcal{B}(f)}_{f \in I},
\]
the $\mathbb{R}_{\max}$-module $Q_M$ is defined as the quotient $\mathbb{R}_{\max}\langle E \rangle/\mathcal{B}(I)$. Its dual 
\[
V_M:=\Hom_{\mathbb{R}_{\max}}(Q_M,\mathbb{R}_{\max})
\]
is known as the tropical linear space associated $M$, and is naturally equipped with $\mathbb{R}_{\max}$-module structure by being the dual module of $Q_M$. In fact, by a tropical linear space, one means a subset of $(\mathbb{R}_{\max})^E$ which arises in this way for some valuated matroid $M$ on $E$. See \cite{bernd}. 

Next, we recall the definition of an algebraic closure operator which will be used throughout the paper. 
 
\begin{mydef}
Let $S$ be a set. An algebraic closure operator on $S$ is a function $\text{cl}:2^S \to 2^S$ satisfying the following conditions:
\begin{enumerate}
    \item 
$A \subseteq \text{cl}(A)$ for any $A \subseteq S$. 
\item 
For $A\subseteq B \subseteq S$, one has $\text{cl}(A) \subseteq \text{cl}(B)$. 
\item 
$\text{cl}( \text{cl}(A)   )=\text{cl}(A)$ for any $A\subseteq S$. 
\item 
$\text{cl}(A) = \bigcup \text{cl}(F)$, where $F$ runs all finite subsets of $A$.
\end{enumerate}
\end{mydef}

\subsection{Valuated matroids}\label{subsection: valuated matroid}

In the following, we recall the definition of a valuated matroid by utilizing $\mathbb{R}_{\max}$. Note that one may employ its algebraic structure to define valuated matroids as well in terms of Grassmann-Pl\"ucker functions as in \cite{baker2019matroids}. 

\begin{mydef}
 	A \emph{valuated matroid} $M$ of rank $d$ on $[n]$: $w:{[n] \choose d} \to \mathbb{R }\cup \{-\infty \}$ such that
		\begin{enumerate}
			\item 
			$w \not \equiv -\infty$,
			\item 
			For any $S,T \in {[n] \choose d}$ and $s \in S\backslash T$, $\exists~t \in T\backslash S$ such that 
			\[
			w(S)+w(T) \leq w((S\backslash \{s\}) \cup \{t\}) +w((T\backslash\{t\}) \cup\{s\}). 
			\]
    where the addition is the usual addition of real numbers. 
		\end{enumerate}   
    
\end{mydef}

The following provides a circuit cryptomorphism for valuated matroids. For a function $X \in (\mathbb{R}\cup \{-\infty\})^E$, we let $\underline{X}$ be the support of $X$, i.e., 
\[
\underline{X}:=\{a \in E \mid X(a) \neq -\infty\}.
\]
We often write $X_a$ for $X(a)$. 

\begin{mythm}\cite{murota2001circuit} \label{theorem: circuits of valuated}
Let $E$ be a finite set and $\mathcal{C} \subseteq (\mathbb{R}\cup \{-\infty\})^E$ satisfying the following conditions:
\begin{enumerate}
    \item[(VC1)]
$(-\infty, \dots, -\infty) \not \in \mathcal{C}$, 
     \item[(VC2)]
If $X,Y \in \mathcal{C}$ with $\underline{X} \neq \underline{Y}$, then $\underline{X} \not \subseteq \underline{Y}$. 
      \item[(VC3)]
For $X \in \mathcal{C}$ and $a \in \mathbb{R}$, $X+(a,\dots,a) \in \mathcal{C}$. 
       \item[(VCE)]
For $X,Y \in \mathcal{C}$ and $u,v \in E$ with $X_u=Y_u \neq -\infty$ and $X_v >Y_v$, $\exists~Z \in \mathcal{C}$ such that 
\begin{equation}
Z_u=-\infty, \quad Z_v=X_v, \quad Z\leq \max(X,Y),
\end{equation}
where $\max(X,Y) \in (\mathbb{R}\cup \{-\infty\})^E$ such that $\max(X,Y)(u)=\max(X_u,Y_u)$.
\end{enumerate}
Then, $\mathcal{C}$ uniquely determines a valuated matroid on $E$, and any valuated matroid arises in this way. 
\end{mythm}

One can easily observe that the condition (VCE) above implies the following circuit elimination ($VCE$'): for $X,Y \in \mathcal{C}$ and $u,v \in E$ with $X_u=Y_u \neq -\infty$ and $v \in \underline{X} \backslash \underline{Y}$, $\exists~Z \in \mathcal{C}$ such that 
\begin{equation}
Z_u=-\infty, \quad Z_v=X_v, \quad Z\leq \max(X,Y).
\end{equation}

\begin{rmk}
 While we typically require $E$ to be finite in our definition of valuated matroids, our results will also have generalizations to infinite valuated matroids.  The notion of infinite valuated matroid we consider is more general than in \cite{dress1992valuated}, for instance we allow valuated matroids of infinite rank (e.g. a linear matroid on a subset of an infinite dimensional space over a finite field. We refer the reader to Appendix \ref{section: infinite val mat} for infinite valuated matroids. 
\end{rmk}

For valuated matroids $M$ and $N$ on $E$ and $E'$ respectively, one can define projective equivalence as follows: there exists a bijection $\phi:E \to E'$ and a map $s:E \to \mathbb{R}_{\max}^\times$ such that
\begin{equation}\label{eq: proj eq}
w_M(B) =  \left(\prod_{i \in B}s(i)\right) w_N(\phi(B)). 
\end{equation}
One can easily see that this induces an isomorphism between the underlying matroids $\underline{M}$ and $\underline{N}$. 

In this paper, we will mainly use the circuit formulation of projective equivalence of valuated matroids as follows.

\begin{mydef}\label{definition: proj eq}
    Let $M$ and $M'$ be valuated matroids on finite sets $E$ and $E'$.  A \emph{projective equivalence} from $M$ to $M'$ consists of a pair $(\phi, s)$ where $\phi: E \rightarrow E'$ is a bijection and $s: E \rightarrow \mathbb{R}_{\max}^\times$ is such that a map $\eta: E \rightarrow\mathbb{R}_{\max}$ is a circuit of $M$ if and only if $\eta':E' \to \mathbb{R}_{\max}$ sending $y$ to $s(\phi^{-1}(y))\eta(\phi^{-1}(y))$ is a circuit of $M'$. 
\end{mydef}

One might instead define a projective equivalence as a bijection $\phi: E \rightarrow E'$ such that their exists some $s$ such that $(\phi, s)$ satisfies our definition.  This is not equivalent to our definition of projective equivalence, but the two definitions agree on which valuated matroids are projectively equivalent.  The notion of a projective equivalence class is all that is used in the statements of our main results, but the definition we give is most convenient for the proof.

\section{The closure operator of a valuated matroid}

Circuits of valuated matroids are typically defined as functions supported on the circuits of the underlying matroid as in Theorem \ref{theorem: circuits of valuated}. However, it will sometimes be more convenient to view circuits as sets. We begin by fixing some notation. 

\begin{notation}
  Let $M$ be a valuated matroid with underlying set $E$.  
\begin{enumerate}
    \item 
$\hat{M}$ denotes the quotient set of $\mathbb{R}_{\max} \times E$ in which $\{-\infty\} \times E$ is collapsed to a point (denoted simply $\textbf{0}$). We will also often denote $-\infty \in \mathbb{R}_{\max}$ by $\textbf{0}$.
\item 
For a subset $C\subseteq \hat{M}\backslash\{\textbf{0}\}$, we define its \emph{support} to be
\[
\mathrm{supp}(C) = \{x\in E\mid \exists a\in \mathbb{R}_{\max}\textrm{ such that }ax\in C\}.
\]
\item 
We denote the equivalence class of $(a, x)\in \hat{M}$ by $ax$. 
\item 
We use $\leq$ to denote the partial order on $\hat{M}$ defined as follows:
\[
ax\leq by \iff \text{either both sides are $\textbf{0}$ or $x=y$ and $a\leq b$.}
\]
\item 
We view $\mathbb{R}_{\max}$ as a semifield, in particular, we use the multiplicative notation for usual addition of real numbers. 
\end{enumerate}   
\end{notation}

Let $T$ be a poset. For any subset $S\subseteq T$, the downward closure $S_\downarrow$ of $S$ (in $T$) is defined as follows. 
\[
S_\downarrow :=\{t\in T\mid \exists s\in S\textrm{ such that }t\leq s\}.
\]

\begin{mydef}  
Let $M$ be a valuated matroid on a finite set $E$. A finite subset $C\subseteq \hat{M}$ is called a \emph{circuit set} if there is some circuit $c: E\rightarrow \mathbb{R}_{\max}$ such that $C = \{c(x)x \mid c(x)\neq \textbf{0}\}$.
\end{mydef}

\begin{myeg}
Let $M=U_{2,4}$ viewed as a valuated matroid. The function $c:\{1,2,3,4\} \to \mathbb{R}_{\max}$ such that 
\[
c(1)=1 \quad c(2)=1, \quad c(3)=1, \quad c(4)=\textbf{0}
\]
is a circuit, and $C=\{(c(1),1), (c(2),1), (c(3),1)\}$ is the corresponding circuit set. 
\end{myeg}

It is easy to see that there is a one-to-one correspondence between circuit sets and circuits. We will use the capital letter $C$ for circuit sets. 

The following lemma is a form of the circuit elimination axiom.

\begin{lem}[Circuit elimination]\label{lemma: circuit elimination}
Let $M$ be a valuated matroid on a finite set $E$.  Let $C_1, C_2\subseteq \hat{M}$ be distinct circuit sets and let $w\in C_1 \cap C_2$.  Then there is some circuit set $C$ such that $C\subseteq (C_1 \cup C_2 \backslash \{w\})_\downarrow$. Moreover, if $\eta \in C_1 \backslash ({C_2})_\downarrow$ then we may choose $C$ so that $\eta \in C$.
\end{lem}
\begin{proof}
We prove the second assertion on $\eta$ which implies the first assertion since either $C_1\not\subseteq (C_2)_\downarrow$ or $C_2\not\subseteq (C_1)_\downarrow$. Let $c_1, c_2: E\rightarrow \mathbb{R}_{\max}$ be the corresponding circuits.  Let $w = kz$ and $\eta = by$ for $(k, z), (b, y)\in \mathbb{R}_{\max}\times E$.  Since $c_1(z) = k = c_2(z)$, the circuit elimination (VCE') gives a circuit $c$ such that
\[
c(z)=\textbf{0}, \quad c(y)=b,\quad \text{and} \quad c(t)\leq \max(c_1(t), c_2(t))\quad \forall t \in E.
\]
Let $C$ be the corresponding circuit set, and let $u=ax\in C$ (i.e. $c(x) = a \neq \textbf{0}$).  Either $c(x)\leq c_1(x)$ or $c(x)\leq c_2(x)$; without loss of generality we assume the former.  Letting $v = c_1(x)x$, we have $v \in C_1$ and $u\leq v$.  Moreover, $x\neq z$ (because $c(x)\neq \textbf{0}$), so $v\neq w$. This shows that $C\subseteq (C_1 \cup C_2 \backslash \{w\})_\downarrow$. For the remaining claim, $\eta\in C$ since $c(y) = b$.
\end{proof}

We now define flats of valuated matroids.

\begin{mydef}\label{definition: flats}
Let $M$ be a valuated matroid on a finite set $E$. A \emph{flat} is a nonempty subset $F\subseteq\hat{M}$ satisfying the following two conditions:
\begin{enumerate}
    \item 
$F$ is downward closed (i.e. $bx\in F$ and $a\leq b$ implies $ax\in F$). 
\item 
For any circuit set $C$ and any $u\in C$, if $C\backslash \{u\}\subseteq F$ then $u\in F$. 
\end{enumerate}
\end{mydef}

\begin{rmk}\label{remark: idyll}
The condition that $F$ is downward closed is similar in spirit to the other condition.  Specifically, it is equivalent to the condition that if the maximum among $a_1, \ldots, a_n$ occurs twice and if all but one element of $\{a_1 x, \ldots, a_n x\}$ lies in $F$ then so does the remaining element.
\end{rmk}

\begin{mydef}\label{definition: closure}
Let $M$ be a valuated matroid on a finite set $E$.
\begin{enumerate}
    \item 
For any $S\subseteq\hat{M}$, we define its \emph{closure} $\mathrm{cl}(S)$ as the intersection of all flats containing $S$.
    \item 
A flat is \emph{finitely generated} if it is the closure of a finite set. 
\end{enumerate}
\end{mydef}

To illustrate the definition, we consider the case of linear valuated matroids.  In this setting we have a vector space $V$ over a valued field $K$ and a finite subset $E\subseteq V$.  A circuit set is a set $\{a_1 x_1, \ldots, a_n x_n\}$ such that there is a linear equation $\sum \alpha_i x_i = 0$ with $v(\alpha_i)=a_i$ and with minimal support (i.e. $x_1, \ldots, x_n$ form a circuit of the underlying matroid). The following proposition shows that in this setting flats are related to $\mathcal{O}_K$-submodules in the same way that flats of ordinary matroids are related to subspaces.

\begin{pro}\label{proposition: flats of linear valuated matroids}
Let $V$ be a vector space over a valued field with surjective valuation $v$ and $M$ the a linear valuated matroid induced by a finite set $E\subseteq V$.  Given an $\mathcal{O}_K$-submodule $N\subseteq V$, let 
\begin{equation}
    F_N = \{ v(c) x \mid x\in E, c\in K, cx\in N\}.
\end{equation} 
Then
\begin{enumerate}
    \item 
$F\subseteq \hat{M}$ is a flat if and only if $F = F_N$ for some $N$.
    \item 
If $F\subseteq \hat{M}$ is a flat, there is exactly one finitely generated $\mathcal{O}_K$-submodule $N\subseteq V$ such that $N$ is generated by elements which have the form $ax$ for $a\in\mathbb{R}_{\max}$ and $x\in E$ and such that $F = F_N$.  In particular, flats of $M$ are in one-to-one correspondence with $\mathcal{O}_K$-submodules of $V$ which are generated by elements of this form.
\end{enumerate}
\end{pro}
\begin{proof}
(1): ($\implies)$ Let $c_1, \ldots, c_n\in K$ and $x_1, \ldots, x_{n + 1}\in E$.  We claim that if $x_{n + 1} = \sum_{i=1}^n c_ix_i$, then
\begin{equation}
x_{n + 1}\in \mathrm{cl}(v(c_1)x_1, \ldots, v(c_n)x_n).    
\end{equation}
For $n = 0$, this is trivial since $\textbf{0}\in\mathrm{cl}(\emptyset)$.  Suppose the claim holds with $n$ replaced by $n - 1$.  Note that $c_1, \ldots, c_n$ are nonzero by the inductive hypothesis, since the claim remains unchanged after omitting terms equal to zero.

If $x_1, \ldots, x_n$ are linearly independent, then $\{x_1, \ldots, x_{n+1}\}$ form a circuit in the underlying ordinary matroid; any smaller dependent set must contain $x_{n + 1}$, and will give a different expression $x_{n + 1} = \sum b_i x_i$, allowing us to cancel the $x_{n+1}$, and contradict independence of $x_1, \ldots, x_n$.  Hence $\{v(c_1)x_1, \ldots, v(c_n)x_n, x_{n + 1}\}$ is a circuit.  Any flat containing $v(c_1)x_1, \ldots, v(c_n)x_n$ also contains $x_{n + 1}$, establishing the claim in this case.

Hence we suppose $x_1, \ldots, x_n$ are linearly dependent.  Without loss of generality, we may assume
\begin{equation}
 c_n x_n = d_1 c_1 x_1 + \ldots + d_{n-1} c_{n-1}x_{n-1}   
\end{equation}
for some $d_1, \ldots, d_{n-1}\in K$.   First, consider the case where $v(d_i) \leq 1$ for all $i$.  We have $x_{n + 1} = \sum_{i=1}^{n-1} (1 + d_i)c_i x_i$.  By the inductive hypothesis, 
\begin{equation}x_{n+1}\in \mathrm{cl}(v((1+d_1)c_1)x_1, \ldots, v((1+d_{n-1})c_{n-1})x_{n-1}).\end{equation} 
Since $v(1 + d_i) \leq 1$ for all $i$, we have $v((1 + d_i)c_i) x_i \in \mathrm{cl}(v(c_1) x_1, \ldots, v(c_{n})x_{n})$ since each flat is downward closed. Consequently 
\begin{equation}
    x_{n+1}\in \mathrm{cl}(v((1+d_1)c_1)x_1, \ldots, v((1+d_{n-1})c_{n-1})x_{n-1})\subseteq\mathrm{cl}(v(c_1) x_1, \ldots, v(c_{n})x_{n}).
\end{equation}

Finally consider the case where there is some $i$ such that $v(d_i) > 1$.  Let $k$ be the index which maximizes $v(d_k)$.  We have
\begin{equation}
    c_k x_k = d_k^{-1} c_n x_n - d_k^{-1} d_1 c_1 x_1 - \ldots - d_k^{-1} d_{n-1}c_{n-1}x_{n-1},
\end{equation}
where the $x_k$ term doesn't appear on the right.  Define, for $i \neq k$ and $i \leq n$, 
\[
b_i:=\begin{cases}
-d_k^{-1}d_ic_i & \text{ if $i<n$},\\
d_k^{-1}c_n  & \text{ if $i=n$}.
\end{cases}
\]
Note that $v(b_i) \leq v(c_i)$ for all $i\neq k$, and consequently $v(b_i + c_i) \leq v(c_i)$.  Since $c_k x_k = \sum_{i\neq k,i\leq n} b_ix_i$, we have
\begin{equation}
    x_{n+1} = \sum_{i\neq k, i\leq n} (c_i + b_i) x_i.
\end{equation}
By the inductive hypothesis and the fact that $v(b_i + c_i)\leq v(c_i)$,
\begin{equation}
    x_{n+1} \in \mathrm{cl}(\{v(c_i + b_i)x_i \mid i\leq n, i\neq k\})\subseteq\mathrm{cl}(\{v(c_i)x_i\mid i\leq n\})
\end{equation}
This establishes the final case of the claim.

 Next observe that by applying the above claim with $c_{n+1}^{-1}c_i$ in place of $c_i$ we see that if $c_{n+1}x_{n+1} = \sum_{i=1}^n c_i x_i$ then $v(c_{n+1}) x_{n + 1}\in \mathrm{cl}(v(c_1)x_1, \ldots, v(c_n)x_n)$.

Let $F$ be a flat, and let $N\subseteq V$ be the set of elements of the form $\sum c_ix_i$ with $c_i\in K$, $x_i\in E$ such that $v(c_i)x_i\in F$ for all $i$.  Observe that if $r\in \mathcal{O}_K$ and $\sum c_i x_i\in N$ (with $c_i, x_i$ as above), then $v(rc_i) \leq v(c_i)$ and the downward closed property of flats ensures $v(rc_i)x_i\in F$, and hence $\sum rc_i x_i \in N$. In particular, $N$ is an $\mathcal{O}_K$-submodule of $V$.  We wish to show $F = F_N$.  If $ax\in F$, by surjectivity of $v$, we may choose $c\in K$ with $v(c) = a$ so $v(c)x\in F$.  Then $cx\in N$, so $F\subseteq F_N$.

For the reverse inclusion, choose an element $ax\in F_N$.  Then $a = v(c_{n+1})$ for some $c_{n+1}\in K$ with $c_{n+1}x\in N$.  By definition of $N$, there exist $c_1, \ldots, c_n\in K$ and $x_1, \ldots, x_n\in E$ such that $c_{n+1}x = \sum_{i=1}^n c_ix_i$ and $v(c_i)x_i\in F$ for $i\leq n$.  By the claim, $ax = v(c_{n+1})x \in \mathrm{cl}(v(c_1)x_1, \ldots, v(c_n)x_n)$.  Since $v(c_i)x_i\in F$ for $i\leq n$, we obtain $ax\in F$.  Hence $F_N = F$ as claimed.

($\impliedby$)  We show that for any $\mathcal{O}_K$-submodule $N\subseteq V$, $F_N$ is a flat. Clearly, $F_N$ is downward closed since the valuation is surjective. Choose a circuit set $\{a_1x_1, \ldots, a_n x_n\}$ with $a_i x_i\in F_N$ for $i < n$.  It suffices to show that $a_n x_n \in F_N$.  By the definition of a linear valuated matroid, there exist $c_1, \ldots, c_n$ with $v(c_i) = a_i$ and $\sum c_i x_i = 0$.  For $i < n$, since $a_ix_i\in F_N$, there exist $d_i$ with $v(d_i)=a_i=v(c_i)$ such that $d_ix_i\in N$.  Since $c_id_i^{-1}\in \mathcal{O}_K^\times$, we also have $c_ix_i\in N$ for $i < n$.  By closure under addition, $-c_n x_n\in N$ and hence $a_nx_n = v(-c_n)x_n\in F_N$ as desired.  

(2): In the proof of surjectivity in part (1), the $N$ we constructed is generated by elements of the form $ax$ with $a\in\mathbb{R}_{\max}$ and $x\in E$.  This shows existence.  Let $N, N'$ be two submodules with generators of this form, and suppose $F_N = F_{N'}$.  It suffices to show the generators of $N$ are contained in $N'$ and hence that $N\subseteq N'$, as the reverse inclusion follows by symmetry.  Suppose $ax\in N$ for some $a\in\mathbb{R}_{\max}$ and $x\in E$.  Then $v(a)x\in F_N = F_{N'}$.  Hence there exists $b\in K$ with $v(b) = v(a)$ and $bx\in N'$.  Since $ab^{-1}\in \mathcal{O}_K^\times$, $ax\in N'$ as desired.
\end{proof}

\begin{lem}Let $M$ be a valuated matroid on a finite set $E$.
\begin{enumerate}
    \item Any intersection of flats is a flat.
    \item A filtered union of flats is a flat.
    \item $\mathrm{cl}$ is an algebraic closure operator on $\hat{M}$.
\end{enumerate}
\end{lem}
\begin{proof}
    The first claim is trivial.  For the second claim let $F = \bigcup F_i$ be a filtered union of flats indexed by some directed set.  It is clear $F$ is downward closed.  Let $C$ be a circuit set, which is necessarily finite.  Let $v\in C$ and suppose $C\backslash \{v\}\subseteq F$.  By finiteness, there is some $i$ with $C\backslash \{v\}\subseteq F_i$.  Then $v\in F_i\subseteq F$.  The third claim follows from the first two by a standard result on closure operators.
\end{proof}

For an ordinary matroid $M$ on $E$, one can describe closure in terms of circuits as follows\footnote{For instance, see \cite[Proposition 1.4.11(2)]{oxley2006matroid}.
}: for $S \subseteq E$
\[
\text{cl}(S)=S \cup \{v \in E \mid \text{$\exists$ a circuit $C$ such that $v \in C$ and $C\backslash \{v\}\subseteq S$ }  \}.
\]
The following lemma is a valuated analogue of the above result.

\begin{lem}\label{lemma: circuits from valuated flats}
Let $M$ be a valuated matroid.  Let $S\subseteq \hat{M}$ be a subset. Let $\mathcal{C}$ be the set of circuit sets.  Then
\begin{equation}\label{equation: closure operation without downward closure}
    \mathrm{cl}(S) = S_\downarrow \cup \{v\in \hat{M} \mid \exists C\in \mathcal{C} \textrm{ such that } v\in C, C\backslash\{v\}\subseteq S_\downarrow \}
\end{equation}
\end{lem}
\begin{proof}
Let $P$ be the right side of \eqref{equation: closure operation without downward closure}.  We first show that $P$ is a flat.  It is clear (by rescaling circuits) that $\{v\in \hat{M} \mid \exists C\in \mathcal{C} \textrm{ s.t. } v\in C, C\backslash\{v\}\subseteq S_\downarrow \}$ is downward closed and hence so is $P$.  For the sake of contradiction, let $C$ be a circuit set and let $w=kz\in C$ and suppose that $C\backslash \{w\}\subseteq P$ but that $w\not\in P$.  We furthermore assume the counterexample $C$ is chosen to minimize $|C \backslash S_\downarrow|$.

We first consider the case where there exists some $u = ax \in C \backslash S_\downarrow$ with $x\neq z$.  By the choice of $C$, we have $u\in P \backslash S_\downarrow$.  By definition of $P$, there is some circuit set $C'$ with $u\in C'$ and $C'\backslash \{u\}\subseteq S_\downarrow$. 

By Circuit elimination (Lemma \ref{lemma: circuit elimination}), we obtain some circuit set $C'' \subseteq ((C \cup C') \backslash \{u\})_\downarrow$ such that $w\in C''$ (note $w\in C$ but $w\not\in C'_\downarrow\subseteq P_\downarrow = P$).  Since $C\backslash \{w\}, C' \subseteq P$, we have 
\[
C'' \backslash \{w\} \subseteq P_\downarrow = P.
\]
Hence, $C''$ is another counterexample to $F$ being a flat.  In order to contradict minimality, we must show that it has fewer elements not contained in $S_\downarrow$ than $C$ does.  

We claim that 
\begin{equation}
 |C'' \backslash S_\downarrow | \leq |((C \cup C') \backslash \{u\}) \backslash S_\downarrow |.   
\end{equation}
In fact,  define a map
\begin{equation}
j: C'' \backslash S_\downarrow \rightarrow ((C \cup C') \backslash \{u\}) \backslash S_\downarrow 
\end{equation}
by choosing $j(by)$ to be some element $b'y\in ((C \cup C') \backslash \{u\}) \backslash S_\downarrow$ such that $by \leq b'y$.  To see this is well-defined, suppose $b'y, b''y\in C\cup C' \backslash\{u\}$ are such that $by \leq b'y, b''y$.  Because a circuit set can not contain two multiples of $y$ we must have $b'y\in C$ and $b''y\in C'$ (or vice versa, in which case we swap $b',b''$).  If $b''y\in C'\backslash \{u\} \subseteq S_\downarrow$, then $by\in S_\downarrow$, contradicting that $by$ is in the domain of $j$.  So $ b''y = u \in C$ and hence $b''y = b'y$.  This establishes $j$ is well-defined.  To establish the claim, we must show $j$ is injective.  Suppose $j(by) = j(b'y')$ for some $by, b'y'\in C''\backslash S_\downarrow$.  Then there exists some $b''y'' = j(by)$ such that $by, b'y' \leq b''y''$ and hence $y = y' = y''$.  But then $by = b'y$ since $C''$ can contain at most one multiple of $y$.  This proves the claim.

Observe that
\begin{equation}
    |C'' \backslash S_\downarrow | \leq |((C \cup C') \backslash \{u\}) \backslash S_\downarrow | \leq |(C \backslash \{u\}) \backslash S_{\downarrow} | + |(C' \backslash \{u\}) \backslash S_\downarrow |.
\end{equation}
Since $C'\backslash\{u\}\subseteq S_\downarrow$, the last term on the right side is zero.  Moreover, $u\not\in S_\downarrow$, so
\begin{equation}
    |C'' \backslash S_\downarrow | \leq |(C \backslash \{u\}) \backslash S_{\downarrow} | = |C \backslash S_{\downarrow} | - 1.
\end{equation}
Thus $|C \backslash S_\downarrow|$ is not minimized, and so by contradiction $P$ is a flat.

In the remaining case any $u\in C\backslash S_\downarrow$ has the form $az$ for some $a\in\mathbb{R}_{\max}$.  A circuit set can only contain one multiple of $z$, and $C$ contains $w=kz$.  Thus $C\backslash S_\downarrow\subseteq \{w\}$.  Since $w\in C$, this implies $w\in P$, which is a contradiction. Hence $P$ is a flat.

Clearly $S\subseteq P$, so it remains to show that if $F'$ is a flat with $S\subseteq F'$ then $P\subseteq F'$.  Let $v\in P$.  If $v\in S_\downarrow$, then $v\in F'$ since $F'$ is downward closed.  Otherwise, there is a circuit set $C$ with $v\in C$ and $C\backslash \{v\}\subseteq S_\downarrow\subseteq F'$.  The latter condition plus the definition of a flat imply $v\in F'$ as desired.  
\end{proof}

The following result lists a few properties of the closure operator on a valuated matroid.

\begin{pro}\label{proposition: valuated closure axioms}
Let $M$ be a valuated matroid on a finite set $E$. Then, $\mathrm{cl}$ is an algebraic closure operator on $\hat{M}$ satisfying the following.
\begin{enumerate}
    \item 
    $\textbf{0}\in\mathrm{cl}(\emptyset)$.
    \item 
    For any $S\subseteq \hat{M}$ and $a\in \mathbb{R}_{\max}^\times$, $\mathrm{cl}(aS) = a \mathrm{cl}(S)$. 
    \item 
    For any $S\subseteq \hat{M}$, $\mathrm{cl}(S)$ is downward closed.
    \item 
    Let $S\subseteq\hat{M}$ and let $x,y\in E$ and $a\in \mathbb{R}_{\max}$.  Suppose there is some $b\in \mathbb{R}_{\max}$ such that $ax\in \mathrm{cl}(S\cup\{by\})$.  Then there is some $c\leq b$ such that $ax\in \mathrm{cl}(S\cup\{cy\})$ and $cy\in \mathrm{cl}(S \cup\{ax\})$.
    \item 
    Let $S\subseteq\hat{M}$ and let $x\in E$ and $a, b\in \mathbb{R}_{\max}$.  Suppose $ax\in \mathrm{cl}(S \cup \{bx\})$ and that $b < a$.  Then $ax\in \mathrm{cl}(S)$.

\end{enumerate}
\end{pro}
\begin{proof}
(1): This is clear since flats are nonempty and downward closed.

(2): We first show that $a\mathrm{cl}(S)$ is a flat. Suppose $C\backslash \{x\} \subseteq a\mathrm{cl}(S)$ for some circuit set $C$ and some $x\in C$.  Then $a^{-1}C \backslash \{a^{-1}x\} \subseteq \mathrm{cl}(S)$.  Since $a^{-1}C$ is a circuit set, $a^{-1}x\in \mathrm{cl}(S)$ so $x\in a\mathrm{cl}(S)$.  Similarly $a\mathrm{cl}(S)$ is downward closed.  This shows $a\mathrm{cl}(S)$ is a flat; since it contains $aS$, $\mathrm{cl}(aS)\subseteq a\mathrm{cl}(S)$.  Applying this with $a^{-1}$ and $aS$ in place of $a$ and $S$ yields $\mathrm{cl}(S)\subseteq a^{-1}\mathrm{cl}(aS)$.  Hence $\mathrm{cl}(aS) = a\mathrm{cl}(S)$, establishing the second claim.

(3): The third claim follows immediately from the definition of a flat.  

(4): Let $S, a, b, x, y$ be as in the fourth claim. We consider the following two cases. 

\underline{Case 1:} Suppose that $ax \in (S\cup \{by\})_\downarrow$. Then $ax \in S_\downarrow$ or $ax \in (by)_\downarrow$. If $ax \in S_\downarrow$, then $ax\in \mathrm{cl}(S)$ and we may take $c=\textbf{0}$. If $ax \in (by)_\downarrow$, then $x=y$ and $a\leq b$. In this case, we can take $c=a$.

\underline{Case 2:} Suppose that $ax\not\in (S \cup \{by\})_\downarrow$. Let 
\begin{equation}
    P = 
    \{v\in \hat{M} \mid \exists C\in \mathcal{C} \textrm{ such that } v\in C, C\backslash\{v\}\subseteq (S\cup\{by\})_\downarrow \},
\end{equation}
where $\mathcal{C}$ is the set of circuit sets.
By Lemma \ref{lemma: circuits from valuated flats}, $ax\in P$.  Therefore, there is a circuit set $C$ with $ax\in C$ and $C\backslash \{ax\} \subseteq (S \cup \{by\})_\downarrow$.  If $C\backslash\{ax\}\subseteq S_\downarrow$ then $ax\in \mathrm{cl}(S)$, and we just set $c=\textbf{0}$.  Hence, we may assume that there exists $c\leq b$ such that $cy\in C$.  Since $cy$ must be the only multiple of $y$ in $C$,
\begin{equation}
    C\backslash \{ax\} \subseteq S_\downarrow\cup \{cy\} \subseteq \mathrm{cl}(S \cup \{cy\}).
\end{equation}
Thus also $ax\in \mathrm{cl}(S \cup \{cy\})$.  Next observe that $C \backslash \{ax, cy\} \subseteq S_\downarrow$ and so 
\begin{equation}
    C \backslash \{cy\} \subseteq S_\downarrow \cup \{ax\} \subseteq \mathrm{cl}(S \cup \{ax\}).
\end{equation}
Hence $cy\in \mathrm{cl}(S \cup \{ax\})$.

(5): Now let $S, a, b, x$ be as in the fifth claim.  Let 
\begin{equation}
    Q = (S\cup\{bx\})_\downarrow \cup \{v\in \hat{M} \mid \exists C\in \mathcal{C} \textrm{ such that } v\in C, C\backslash\{v\}\subseteq (S\cup\{bx\})_\downarrow \}.
\end{equation}
Then $ax\in Q$ by Lemma \ref{lemma: circuits from valuated flats}.  By assumption, $ax\not\in \{bx\}_\downarrow$.  If $ax\in S_\downarrow$ then $ax\in\mathrm{cl}(S)$.  Thus we may assume there is a circuit set $C$ such that $ax\in C$ and $C\backslash\{ax\}\subseteq (S\cup\{bx\})_\downarrow$.  Since a circuit set cannot contain two multiples of $x$, $C\backslash\{ax\}\subseteq S_\downarrow\subseteq\mathrm{cl}(S)$.  This implies $ax\in \mathrm{cl}(S)$.
\end{proof}

The following lemma describes circuit sets in terms of the closure operator of a valuated matroid.

\begin{lem}\label{lemma: closure operator determines valuated circuits}
Let $M$ be a valuated matroid on a finite set $E$.  A nonempty finite subset $C\subseteq\hat{M}$ is a circuit set if and only if it satisfies all of the following.
\begin{enumerate}
    \item 
$v\in \mathrm{cl}(C\backslash\{v\})$ for any $v\in C$.
    \item
$\textbf{0}\not\in C$ and if $ax, bx\in C$ then $a = b$.
    \item 
Let $D\subseteq C_\downarrow$ satisfy (1) and (2) and suppose $C \cap D\neq \emptyset$.  Then $D = C$.
\end{enumerate}
\end{lem}
\begin{proof}
It is clear that circuit sets satisfy Conditions (1)  and (2).  To check they satisfy (3), let $C$ be a circuit set, let $D\subseteq C_\downarrow$ satisfy Conditions (1) and (2) and let $v\in C \cap D$. Then, we have $v \in \text{cl}(D\backslash \{v\})$ by Condition $(1)$. Since $v\not\in (D \backslash \{v\})_\downarrow$, by Lemma \ref{lemma: circuits from valuated flats} there is a circuit set $C'$ with $v\in C'$, and 
\begin{equation}\label{eq: inclusion}
C' \backslash \{v\} \subseteq (D \backslash \{v\})_\downarrow\subseteq (C\backslash\{v\})_\downarrow.
\end{equation}
Here we have used that since Condition (2) applies to $D$, $(D\backslash\{v\})_\downarrow \cap \{v\}_\downarrow = \emptyset$. By Lemma \ref{lemma: circuit elimination}, if $C\neq C'$ then there is a circuit set $C'' \subseteq ((C \cup C') \backslash \{v\})_\downarrow$.  In particular, by \eqref{eq: inclusion}, we have
\begin{equation}
C'' \subseteq (C\backslash\{v\})_\downarrow.
\end{equation}
Let $\pi: \hat{M}\backslash\{\textbf{0}\} = \mathbb{R}_{\max}^\times\times E\rightarrow E$ be the projection.  Then, we have
\begin{equation}
\pi(C'') \subseteq \pi(C) \backslash \{\pi(v)\} \subsetneq \pi(C).
\end{equation}
But $\pi(C'')$ and $\pi(C)$ are circuits of the underlying matroid, so this is a contradiction, so $C' = C$.  Consequently, \eqref{eq: inclusion} becomes the following:
\begin{equation}
C \backslash \{v\} \subseteq (D \backslash\{v\})_\downarrow \subseteq (C\backslash\{v\})_\downarrow.    
\end{equation}
For any $w\in C\backslash\{v\}$, there exist $w'\in D\backslash\{v\}$ and $w''\in C\backslash\{v\}$ such that $w\leq w' \leq w''$.  By Condition (2), we have equality throughout, and hence $C\backslash\{v\}\subseteq D\backslash\{v\}$.  A similar argument using that Condition (2) applies to $D$ gives $D\backslash\{v\}\subseteq C\backslash\{v\}$.  Thus $C \backslash\{v\} = D\backslash\{v\}$. and hence $C = D$.

Conversely, let $C$ be a nonempty subset satisfying Conditions (1), (2), and (3).  Let $v\in C$ and observe $v\in \mathrm{cl}(C \backslash \{v\})$.
By Lemma \ref{lemma: circuits from valuated flats}, either there is a circuit set $C'$ with $v\in C'$ and $C'\backslash\{v\} \subseteq (C\backslash\{v\})_\downarrow$ or we have $v\in (C \backslash \{v\})_\downarrow$. The second case is not possible by Condition (2). In the first case, $C'$ satisfies Conditions (1)  and (2) because it is a circuit set, and clearly $C'\subseteq C_\downarrow$ and $v\in C \cap C'\neq \emptyset$.  So by Condition (3), $C = C'$ and hence $C$ is a circuit set.
\end{proof}

To understand the intuition behind the following lemma, it is helpful to consider instead the case of a classical matroid (i.e. replace $\mathbb{R}$ with the trivial group), and to assume $\mathrm{cl}$ is the closure operator of a matroid.  In combination with the description of circuits from Lemma \ref{lemma: closure operator determines valuated circuits}, the below lemma is a valuated analogue of the fact that a dependent set (defined in terms of the closure operator) contains a circuit. To this end, we introduce the following definition. 

\begin{mydef}\label{definition: vector set}
Let $E$ be a finite set.  Let $\mathrm{cl}$ be an algebraic closure operator on $\hat{M}$ satisfying the conclusions of Proposition \ref{proposition: valuated closure axioms}.  We call a finite subset $V\subseteq \hat{M}$ a \emph{vector set} if it satisfies the following properties:
\begin{enumerate}
    \item 
for any $v\in V$,  $v\in \mathrm{cl}(V\backslash\{v\})$, 
\item 
$\textbf{0}\not\in V$, and 
\item 
if $ax,bx\in V$ for $x\in E$, then $a = b$. 
\end{enumerate}
\end{mydef}

\begin{lem}\label{lemma: bend relations from element of closure of valuated matroid}
Let $E$ be a finite set.  Let $\mathrm{cl}$ be an algebraic closure operator on $\hat{M}$ satisfying the conclusions of Proposition \ref{proposition: valuated closure axioms}.  Let $S\subseteq \hat{M}$ be finite and let $v = ax\in S$ be such that $v\in \mathrm{cl}(S\backslash v)$ and such that $bx\not\in S$ for all $b > a$.  Then there is a vector set $T\subseteq S_\downarrow$ with $v\in T$. 
\end{lem}
\begin{proof}
Fix $v=ax$ and suppose $S$ is a counterexample of minimal size.  We first claim that if $by, cy\in S$ for some $b\neq c$, then $x = y$. In fact, suppose $by, cy\in S$ for some $y\in E\backslash\{x\}$ with $b\neq c$. We may assume that $b\leq c$, and hence 
\[
v\in \mathrm{cl}(S\backslash \{v\}) = \mathrm{cl}(S \backslash \{v, by\}).
\]
Since $S' := S \backslash \{by\}$ is smaller than $S$, we have a vector set $T \subseteq S'_\downarrow$ with $v \in T$. But, then $T \subseteq S_\downarrow$ as well, giving us contradiction. 

Next, we claim that $bx \not \in S$ for $b \neq a$. In fact, suppose $bx\in S$ for some $b\neq a$.  By hypothesis on $S$, $b\not > a$, so $b < a$.  Now, 
\[
ax\in \mathrm{cl}((S\backslash\{ax,bx\}) \cup \{bx\}).
\]
Thus by the fifth axiom of Proposition \ref{proposition: valuated closure axioms}, we have
\[
ax\in\mathrm{cl}((S\backslash\{bx\})\backslash\{ax\}),
\]
and hence $S\backslash\{bx\}$ satisfies the hypotheses on $S$. By minimality of the counterexample $S$, we obtain a vector set $T\subseteq (S\backslash\{bx\})_\downarrow\subseteq S_\downarrow$ with $v \in T$. This shows contradiction.

Let $m = |S|$.  We will work with a subset $T = \{v_1, \ldots, v_m\}\subseteq S_{\downarrow}$ such that $v\in T$ and no two elements $v_i$ and $v_j$ of $T$ are of the form $v_i=bz$ and $v_j=cz$.
We will also work with some choice of ordering of the elements of $T$ so that $v_1 = v$.

We will define $k$ to be the largest number such that $v_i \in \mathrm{cl}(v_1, \ldots, v_{i-1}, v_{i + 1}, \ldots, v_m)$ for $i = 1, \ldots, k$.  Note that if we chose $T = S$ we would have $k\geq 1$, but our goal is to construct such a $T$ with $k=|T|$. We may choose $T$ (and the ordering of the elements of $T$ other than $v_1 = v$) so that $k$ is maximal (in particular $k\geq 1$) and we assume for the sake of contradiction that $k \neq |T|$.

For any $i\leq k$, by the fourth axiom of Proposition \ref{proposition: valuated closure axioms}, there is some $c_i \leq 1$ such that 
\begin{equation}\label{equation: inclusion of c_iv_k}
    c_i v_{k+1} \in \mathrm{cl}(v_1, \ldots, v_k, v_{k + 2}, \ldots, v_m)
\end{equation}
and
\begin{equation}\label{equation: inclusion of v_i into a closure involving c_iv_k}
    v_i \in \mathrm{cl}(v_1, \ldots, v_{i - 1}, v_{i + 1}, \ldots, c_i v_{k + 1}, \ldots, v_m).
\end{equation}

Let $i_0$ be the index such that $c_{i_0}=\max \{c_i\}$. 
Let $c = c_{i_0}$.  Then by \eqref{equation: inclusion of c_iv_k} applied to $i_0$, we have 
\begin{equation}
    c v_{k+1} \in \mathrm{cl}(v_1, \ldots, v_k, v_{k + 2}, \ldots, v_m).
\end{equation}
Furthermore, for $i\leq k$, we have $c_i\leq c$, and hence $c_iv_{k+1}\in \mathrm{cl}(cv_{k+1})$. Therefore, we have
\begin{equation}
   v_{i} \in \mathrm{cl}(v_1, \ldots, v_{i - 1}, v_{i + 1}, \ldots, c v_{k + 1}, \ldots, v_m).
\end{equation}
Let $T':=\{v_1, \ldots, v_k, cv_{k+1}, v_{k+2}, \ldots, v_m\}\subseteq S_{\downarrow}$.  If $v_{k + 1}$ is a multiple of $y\in E$, then no other element of $T$ is a multiple of $y$, so $cv_{k + 1}\not\in\{v_1, \ldots, v_{k}, v_{k + 2},\ldots, v_m\}$.  Thus the elements of $T'$ are distinct and $|T'| = m$. Clearly $T'$ satisfies the assumptions on $T$ but with $k + 1$ in place of $k$.  This contradicts maximality.
This yields the desired result except that we may have $\textbf{0}\in T$.  In this case, $T\backslash\{\textbf{0}\}$ has all of the desired properties.
\end{proof}

Our next goal is to show the closure axioms of Proposition \ref{proposition: valuated closure axioms} characterize valuated matroids.  The main step in this direction is to show they imply a circuit elimination property.  One step towards this is the following vector elimination property.

\begin{lem}\label{lemma: vector elimination property}
Let $E$ be a finite set.  Let $\mathrm{cl}$ be an algebraic closure operator on $\hat{M}$ satisfying the conclusions of Proposition \ref{proposition: valuated closure axioms}. Let $V_1, V_2 \subseteq \hat{M}$ be vector sets. Let $e\in V_1\cap V_2$, and let $u$ be a maximal element (with respect to the partial order on $\hat{M}$) of $V_1 \Delta V_2$.  Then there is a vector set $V\subseteq (V_1 \cup V_2 \backslash\{e\})_\downarrow$ containing $u$.
\end{lem}
\begin{proof}
Without loss of generality, $u\in V_1\backslash V_2$.  If $u< w$ for $w\in V_1 \cup V_2 \backslash\{e\}$, then $w\not\in V_1$ by Condition (3) of Definition \ref{definition: vector set}. Hence, $w\in V_1\Delta V_2$, and by maximality of $u$, $w = u$.  Thus $u$ is even maximal in $V_1 \cup V_2 \backslash\{e\}$.

Now, since $e \in \mathrm{cl}(V_2 \backslash \{e\})$, we have
\begin{equation}
\mathrm{cl}(V_1 \backslash\{u\}) = \mathrm{cl}(V_1\backslash\{u, e\}) \vee \mathrm{cl}(e) \subseteq \mathrm{cl}(V_1 \backslash \{u, e\} \cup V_2 \backslash\{e\}) = \mathrm{cl}((V_1 \cup V_2\backslash\{e\})\backslash\{u\}).
\end{equation}
Since, we have $u\in \mathrm{cl}(V_1 \backslash\{u\})$, we conclude that 
\[
u \in \mathrm{cl}((V_1 \cup V_2\backslash\{e\})\backslash\{u\}).
\]
Maximality of $u$ ensures we can apply Lemma \ref{lemma: bend relations from element of closure of valuated matroid} to $S=(V_1\cup V_2 )\backslash \{e\}$ and $v=u$, to see that there is a subset $V\subseteq (V_1\cup V_2\backslash\{e\})_{\downarrow}$, which is a vector set containing $u$.
\end{proof}

\begin{lem}\label{lemma: rescaling lemma}
Let $V, V'\subseteq \hat{M}$ be finite subsets, and assume $\textbf{0}\not\in V \cup V'$.  Suppose $\mathrm{supp}(V')\subseteq \mathrm{supp}(V)$.  Then there exists $c\in \mathbb{R}_{\max}^\times$ such that $cV' \cap V\neq \emptyset$ and $cV'\subseteq V_\downarrow$.
\end{lem}
\begin{proof}
Since $\mathrm{supp}(V')\subseteq \mathrm{supp}(V)$, we may write $V = \{a_1x_1, \ldots, a_nx_n\}$ and $V' = \{b_1x_1, \ldots, b_kx_k\}$ for some $k \leq n$ and some $a_i, b_i\in \mathbb{R}_{\max}^\times$ and $x_i\in E$.  Set $c = \min_i b_i^{-1}a_i$.  Then $cb_i\leq a_i$ for all $i$ and there is some $i$ where equality is attained.  It is easy to see that $cV' = \{cb_1x_1, \ldots, cb_kx_k\}$ has the stated properties.
\end{proof}

The following lemma is another step towards proving the circuit elimination property from the valuated closure axioms.  The following lemma will essentially allow us to extract a smaller circuit set containing a specified element $u$ from a vector set containing $u$.  Circuit elimination will then be done by using Lemma \ref{lemma: vector elimination property}, and applying Lemma \ref{lemma: circuit set inside vector set} to the resulting vector set.

\begin{lem}\label{lemma: circuit set inside vector set}
Let $E$ be a finite set.  Let $\mathrm{cl}$ be an algebraic closure operator on $\hat{M}$ satisfying the conclusions of Proposition \ref{proposition: valuated closure axioms}.  Let $S\subseteq \hat{M}$ and $u\in \hat{M}$.  Suppose there is a vector set $V\subseteq S_\downarrow$ such that $u\in V$.  Then there is a vector set $C\subseteq S_\downarrow$ such that $u\in C$ and such that $C$ has minimal support among all vector sets.
\end{lem}
\begin{proof}
Among all vector sets $V$ with $u\in V\subseteq S_\downarrow$ choose $V$ to have minimal support.  If $V$ has minimal support among vector sets with $V\subseteq S_\downarrow$, then $V$ has minimal support among all vector sets so the claim holds.  Hence we may assume otherwise; there exists a vector set $V'\subseteq S_\downarrow$ with $\mathrm{supp}(V') \subsetneq \mathrm{supp}(V)$.  For any $c\in \mathbb{R}_{\max}^\times$, it is easy to see that $cV'$ is also a vector set and also has the stated properties.  Since $\mathrm{supp}(V')\subseteq \mathrm{supp}(V)$, by Lemma \ref{lemma: rescaling lemma} we may assume $V'\subseteq V_\downarrow$ and $V \cap V'\neq \emptyset$. Let $e'\in V\cap V'$. 

    By minimality of $V$, $u\not\in V'$, so $u\in V \Delta V'$.  To show $u$ is maximal in $V\Delta V'$, let $w\in V\Delta V'$ with $u\leq w$.  If $w\in V$ then $u = w$ by the third condition of a vector set, so assume $w\in V'\backslash V \subseteq V_\downarrow$.  Then there exists $w'\in V$ with $u\leq w \leq w'$.  But again by the third condition of a vector set, we have equality throughout, establishing maximality.
    
    Since $V, V'$ are vector sets with $e'\in V\cap V'$ and such that $u\in V\Delta V'$ is maximal, we may apply Lemma \ref{lemma: vector elimination property} to obtain a vector set $V''\subseteq (V \cup V' \backslash \{e'\})_\downarrow$ such that $u\in V''$.  Since $V\subseteq S_\downarrow$ and $V'\subseteq V_\downarrow$, we get that $V''\subseteq S_\downarrow$.  Write $e'=cz$ with $(c,z)\in \mathbb{R}_{\max}\times E$.  Then
    \begin{equation}
        \mathrm{supp}(V'') \subseteq \mathrm{supp}(V \cup V' \backslash \{e'\}) = \mathrm{supp}(V) \cup \mathrm{supp}(V') \backslash \{z\} = \mathrm{supp}(V) \backslash\{z\}.
    \end{equation}
    This contradicts minimality of $V$, establishing the claim.
\end{proof}

\begin{mythm}\label{proposition: valuated matroid closure cryptomorphism}
Let $E$ be a finite set.  Let $\mathrm{cl}$ be an algebraic closure operator on $\hat{M}$ satisfying the conclusions of Proposition \ref{proposition: valuated closure axioms}.  Then there is a unique valuated matroid structure $M$ with underlying set $E$ such that $\mathrm{cl}$ is the closure operator associated with $M$.
\end{mythm}
\begin{proof}
Call $C\subseteq \hat{M}$ a \emph{circuit set} if it is a vector set with minimal support.  

Let $C_1, C_2$ be circuit sets and $e\in C_1\cap C_2$.  Let $u$ be a maximal element of $C_1\Delta C_2$ under the partial order on $\hat{M}$. By Lemma \ref{lemma: vector elimination property}, there is a vector set $V\subseteq (C_1 \cup C_2\backslash\{e\})_\downarrow$ with $u\in V$. By Lemma \ref{lemma: circuit set inside vector set}, there is a circuit set with $u\in C \subseteq (C_1 \cup C_2 \backslash\{e\})_\downarrow$.

Thus we have the elimination property (Lemma \ref{lemma: circuit elimination}) for circuit sets.  Since the more standard characterization of valuated matroids uses valuated circuit maps, in order to show we have constructed a valuated matroid, we must rephrase the above in terms of valuated circuit maps.

For each circuit set $C$, define $\psi_C: E \rightarrow \mathbb{R}_{\max}$ by $\psi_C(x) = a$ if $ax\in C$ and $\psi_C(x) = -\infty$ if $ax\not\in C$ for all $a$.  Maps $\psi_C$ with $C$ a circuit set will be called \emph{valuated circuits}.  The third condition of Definition \ref{definition: vector set} implies $\psi_C$ is well-defined.  The fact that $\textbf{0}\not\in C$ for any circuit set ensures the zero map is not a valuated circuit.  Strong elimination for valuated circuits is a simple rephrasing of the corresponding property of circuit sets (Lemma \ref{lemma: circuit elimination}).  It is easy to see that if $C$ is a circuit set and $a\in\mathbb{R}_{\max}^\times$ then $aC$ is a circuit set, and hence $a\psi_C = \psi_{aC}$ is a valuated circuit.  Since $\supp(\psi_C) = \supp(C)$, it is clear from the minimality in the definition of circuit sets that we can never have $\supp(\psi_{C_1})\subsetneq \supp(\psi_{C_2})$.  Thus all axioms of a valuated matroid are satisfied.

Next we show that the closure operator $\mathrm{cl}'$ associated to this valuated matroid is $\mathrm{cl}$. Let $S\subseteq\hat{M}$.  Suppose $v\in \mathrm{cl}'(S)$.  Then either $v\in S_\downarrow\subseteq\mathrm{cl}(S)$ or there is a circuit set $C$ such that $v\in C$ and $C \backslash \{v\}\subseteq S_\downarrow\subseteq\mathrm{cl}(S)$.  In the latter case, by the definition of a circuit set, we have $v\in \mathrm{cl}(C \backslash \{v\}) \subseteq \mathrm{cl}(S)$.

Conversely, suppose $v = ax\in \mathrm{cl}(S)$.  If no element of $S$ has the form $bx$ for $b > a$, then using Lemma \ref{lemma: bend relations from element of closure of valuated matroid}, we obtain a vector set $V$ with $v\in V$ and $V \backslash \{v\}\subseteq S_\downarrow$.  By Lemma \ref{lemma: circuit set inside vector set}, there is a circuit set $C$ with $v\in C$ and $C\backslash\{v\}\subseteq C \subseteq (S\cup \{v\})_\downarrow$.  Since $C$ can not contain two multiples of $x$, we must have $C\backslash\{v\}\subseteq S_\downarrow$. Then $v\in \mathrm{cl}'(S)$.  If $bx\in S$ for $b \geq a$, then $v\in S_\downarrow\subseteq\mathrm{cl}'(S)$. 

We have a map from valuated matroids on $E$ to algebraic closure operators on $\hat{M}$ satisfying the conclusions of Proposition 3.11.  So far, we have constructed a right inverse to this map and hence have established surjectivity. It suffices to show injectivity, i.e. that if two valuated matroids on $E$ have the same closure operator then they are equal. This follows either from Lemma \ref{lemma: closure operator determines valuated circuits} or from the characterization of covectors in terms of flats that appears below.
\end{proof}

We make several remarks below. 

\begin{rmk}\label{remark: checking axioms on finite sets}
When checking the axioms of Proposition \ref{proposition: valuated closure axioms}, one can consider only closures of finite sets.  To illustrate the technique, consider an algebraic closure operator which satisfies axiom (4) on finite sets.  Let $S\subseteq\hat{M}$ be infinite and suppose $ax\in \mathrm{cl}(S\cup\{by\})$.  Since $\mathrm{cl}$ is algebraic, there is a finite subset $S'$ such that $ax\in \mathrm{cl}(S'\cup\{by\})$.  Then by applying axiom (4) to the finite set $S'$, we get some $c \leq b$ such that $ax\in\mathrm{cl}(S'\cup \{cy\})\subseteq\mathrm{cl}(S\cup\{cy\})$ and $cy\in\mathrm{cl}(S'\cup\{ax\})\subseteq\mathrm{cl}(S\cup\{ax\})$.  Hence the axiom holds for all infinite sets.
\end{rmk}

\begin{rmk}
Theorem \ref{proposition: valuated matroid closure cryptomorphism} also works over $\mathbb{B}$ instead of $\mathbb{R}_{\max}$.  This gives a cryptomorphism of ordinary matroids in terms of closure operators on $E \cup \{0\}$ instead of $E$.  This is easily remedied as follows.

Given a set $X$ and an element $\mathbf{0}\not\in X$, algebraic closure operators $\mathrm{cl}$ on $X \cup \{\mathbf{0}\}$ with $\mathbf{0}\in\mathrm{cl}(\emptyset)$ are in one-to-one correspondence with algebraic closure operators $\mathrm{cl}'$ on $X$ with the correspondence given by $\mathrm{cl}'(S) = \mathrm{cl}(S) \backslash\{\mathbf{0}\}$ and $\mathrm{cl}(S) = \mathrm{cl}'(S\backslash\{\mathbf{0}\}) \cup \{\mathbf{0}\}$.  Using this correspondence, one may give a cryptomorphic axiomatization of valuated matroids on $E$ in terms of closure operators on $\mathbb{R} \times E \cong \hat{M} - \{\mathbf{0}\}$ satisfying certain axioms.  We leave the details to the interested reader.

Besides omitting condition (1) of Proposition \ref{proposition: valuated closure axioms}, the other main difference in the axiomatization of $\mathrm{cl}'$ comes from condition (4) of Proposition \ref{proposition: valuated closure axioms}.  To see the equivalent condition on $\mathrm{cl}'$ to condition (4), it is helpful to rephrase (4) in a way that doesn't use the element $\mathbf{0}$.  It suffices to check condition (4) when $\mathbf{0}\not\in S$, and condition (4) is trivial if $a = -\infty$ or $b = -\infty$.  But we must consider cases according to whether $c = -\infty$.  One may check that when $c = -\infty$, the conclusion of condition (4) is equivalent to $ax\in \mathrm{cl}(S)$.  Hence condition (4) is equivalent to the condition that for $S\subseteq \hat{M} - \mathbf{0}$, $a, b\in \mathbb{R}_{\max}^\times$ and $x, y\in E$, if $ax\in \mathrm{cl}(S \cup \{by\})\backslash\mathrm{cl}(S)$ then there is some $c\in \mathbb{R}_{\max}^\times$ with $c\leq b$ such that $ax\in \mathrm{cl}(S\cup \{cy\})$ and $cy\in \mathrm{cl}(S\cup\{ax\})$.  Using that $\mathrm{cl}(T) = \mathrm{cl}'(T) \cup \{\mathbf{0}\}$ for all $T$ not containing $\mathbf{0}$, we see that $\mathrm{cl}$ satisfies the above condition (hence condition (4)) if and only if $\mathrm{cl}'$ satisfies this condition.
\end{rmk}

Next we turn to establishing a link between flats and covectors.  For this purpose we will consider covectors with values in an $\mathbb{R}_{\max}$-module rather than in $\mathbb{R}_{\max}$ itself.

\begin{mydef}
Let $M$ be a valuated matroid on a finite set $E$.  Let $A$ be an $\mathbb{R}_{\max}$-module.  An \emph{$A$-valued covector} is a function $f: E\rightarrow A$ such that for any circuit $c$ and any $x\in E$,
\begin{equation}
    \bigvee_{y\in E} c(y) f(y) = \bigvee_{y\in E\backslash\{x\}} c(y) f(y).
\end{equation}
\end{mydef}

\begin{rmk}
Let $M$ be a valuated matroid on a finite set $E$. Let $Q_M$ be the quotient of the free $\mathbb{R}_{\max}$-module generated by $E$ by relations
\begin{equation}
    \bigvee_{y\in E} c(y) y = \bigvee_{y\in E\backslash\{x\}} c(y) y
\end{equation}
for every circuit $c$ and every $x\in E$. This module was introduced in \cite{giansiracusa2018grassmann}.  The obvious map $i: E\rightarrow Q_M$ is then a $Q_M$-valued covector.  In fact, by the universal properties of free and quotient modules, it is the universal covector in the sense that for any $A$-valued covector $f: E\rightarrow A$ (with $A$ an $\mathbb{R}_{\max}$-module), there is a unique module homomorphism $\phi: Q_M\rightarrow A$ such that $f =\phi i$.  By this universal property, we recover the known result that the dual module $V_M = \mathrm{Hom}(Q_M, \mathbb{R}_{\max})$ (commonly called the tropical linear space of $M$) may be identified with the set of $\mathbb{R}_{\max}$-valued covectors.
\end{rmk}

\begin{rmk}
It has been suggested (e.g. in the introduction of \cite{fink2025extensions}) that $V_M$ should play the role of the meet-semilattice of finitely generated flats in the valuated setting.  Similarly, \cite{hirai2019uniform} studies the subset $V_M \cap \mathbb{Z}^n$ from a lattice-theoretic perspective.  An alternative perspective comes from \cite[Theorem 1.1]{crowley2020module}, which identifies $Q_M$ with the join-semilattice of flats in the ordinary matroid setting.  Over the Boolean semifield $\mathbb{B}$, module duality is the same as lattice duality, so these perspectives are equivalent.  However, over $\mathbb{R}_{\max}$ this is no longer true.  It will turn out that the interpretation of $Q_M$ in terms of flats has a generalization in terms of our definition of flats of valuated matroids.  However, $V_M$ will be the module-theoretic dual rather than the lattice-theoretic dual of the join semilattice of finitely generated flats.
\end{rmk}

As the remark hints, our next goal will be to identify $Q_M$ with the join semilattice of finitely generated flats.

\begin{lem}
Let $M$ be a valuated matroid on a finite set $E$. Let $\mathcal{S}$ be the set of finitely generated flats of $M$. Then $\mathcal{S}$ has an $\mathbb{R}_{\max}$-module structure where sum is join and scalar multiplication is $aF:=\{av \mid v \in F\}$. 
\end{lem}
\begin{proof}
The only nontrivial part is the distributivity. But, one has
\[
a\text{cl}(S)  + a \text{cl}(T)= \text{cl}(aS  \cup aT) = a(\text{cl}(S) + \text{cl}(T)),
\]
showing one distributive law. For the other, suppose without loss of generality that $a\leq b$.
\begin{equation}
    (a + b)\mathrm{cl}(S) = b\mathrm{cl}(S) = \mathrm{cl}(bS) = \mathrm{cl}((a + b)S).
\end{equation}
\end{proof}

\begin{pro}\label{proposition: iso $Q_M$ to flats}
Let $M$ be a valuated matroid on a finite set $E$.  Let $\mathcal{S}$ be the $\mathbb{R}_{\max}$-module of finitely generated flats.  Then $\mathcal{S}$ is canonically isomorphic to $Q_M$.  In particular the map $\iota: E\rightarrow\mathcal{S}$ given by $\iota(x) = \mathrm{cl}(\{x\})$ is the universal covector.
\end{pro}
 \begin{proof}
We prove $\mathcal{S}$ satisfies the universal property which implies that $\mathcal{S}$ is isomorphic to $Q_M$.

Let $A$ be an $\mathbb{R}_{\max}$-module and $f:E \rightarrow A$ be a covector.  We must show there is a unique homomorphism $\hat{f}: \mathcal{S}\rightarrow A$ such that $f = \hat{f}\iota$, i.e. $f(x) = \hat{f}(\mathrm{cl}(\{x\}))$.  In order for $\hat{f}$ to be a homomorphism, we must have 
\begin{equation}
    \hat{f}(\mathrm{cl}(\{a_1x_1, \ldots, a_nx_n\})) = \hat{f}(\bigvee_i a_i \mathrm{cl}(\{x_i\})) = \bigvee a_i f(x_i).
\end{equation}
This establishes uniqueness.  Define $\hat{f}$ via this formula.  

To show $\hat{f}$ is well-defined, we claim that for $F=\mathrm{cl}(\{a_1x_1, \ldots, a_nx_n\})\in \mathcal{S}$, we have
\begin{equation}
\hat{f}(F) = \bigvee_{cz\in F} cf(z).    
\end{equation}
In fact, since $a_ix_i\in F$ for all $i$, we have
\begin{equation}
\hat{f}(F) \leq \bigvee_{cz\in F} cf(z).    
\end{equation}
We want to show that for $cz\in F$, we have $cf(z)\leq \hat{f}(F)$.  Since $cz\in F = \mathrm{cl}(\{a_1x_1, \ldots, a_nx_n\})$, by Lemma \ref{lemma: circuits from valuated flats}, either $cz\leq a_ix_i$ for some $i$ or there exists a circuit set $C$ with $cz\in C$ and $C \backslash\{cz\} \subseteq \{a_1x_1, \ldots, a_nx_n\}_\downarrow$.  In the first case, $cf(z)\leq a_if(x_i)\leq \hat{f}(F)$, so we may assume the second case.  By reordering the $a_i$ and $x_i$, we may assume there exist $k$ and $b_1, \ldots, b_k\in\mathbb{R}_{\max}$ such that $b_i\leq a_i$ for $i\leq k$ and $C = \{cz, b_1x_1, \ldots, b_kx_k\}$.  Since $f$ is a covector, we have
\begin{equation}
    cf(z)\leq cf(z) \vee \bigvee_{i\leq k} b_if(x_i) = \bigvee_{i\leq k} b_if(x_i) \leq \bigvee_{i\leq k} a_i f(x_i) \leq \bigvee_{i\leq n}a_if(x_i) = \hat{f}(F).
\end{equation}
Hence $\hat{f}$ is well-defined.

We clearly have $\hat{f}(\mathrm{cl}(\{x\})) = f(x)$.  Let $S = \{a_1x_1, \ldots, a_nx_n\}$ and $S' = \{b_1y_1, \ldots, b_my_m\}$.  Then 
\begin{equation}
\hat{f}(\mathrm{cl}(S)) \vee \hat{f}(\mathrm{cl}(S'))
= \bigvee a_if(x_i) \vee \bigvee b_jf(y_j)
= \hat{f}(\mathrm{cl}(S \cup S'))
= \hat{f}(\mathrm{cl}(S) \vee \mathrm{cl}(S')).
\end{equation}
Thus $\hat{f}$ preserves addition.  That $\hat{f}(a\mathrm{cl}(S)) = a\hat{f}(\mathrm{cl}(S))$ is trivial.  Hence $\hat{f}$ has the required properties.
\end{proof}

\section{Simple valuated matroids as Geometric $\mathbb{R}_{\max}$-modules}

Simple matroids are equivalent to geometric lattices, i.e. lattices that are atomic and semimodular.  Since finite lattices are the same as finitely generated $\mathbb{B}$-modules, it is natural to look for analogous conditions for $\R_{\max}$-modules. Rather than stating things only for $\mathbb{R}_{\max}$-modules, since there is no added difficulty, we consider an arbitrary idempotent semifield or its valuation semiring (defined below) for much of what follows.
 
To this end, in this section we will consider three conditions, namely, (1) atomicity, (2) semimodularity, and (3) (strongly) torsion-freeness. The conditions (1) and (2) are direct translations of the same concepts in lattice theory to module theory over an idempotent semifield. The third condition is a new condition which is not used in lattice theory (or for modules over $\mathbb{B}$). By a geometric $\mathbb{R}_{\max}$-module, we mean an $\mathbb{R}_{\max}$-module satisfying all three conditions.

\begin{mydef}
Let $K$ be an idempotent semifield.  We define its \emph{valuation subsemiring} $\mathcal{O}_K$ as $\{x\in K\mid x\leq 1_K\}$.\footnote{One can easily check that $\mathcal{O}_K$ is indeed a subsemiring.}
\end{mydef}

\begin{mydef}
Let $R$ be an idempotent semiring and $N$ be an $R$-module.  $N$ is \emph{torsion-free} if $ax = bx$ (with $a,b\in R$ and $x\in N$) implies $a = b$ or $x = 0$.
\end{mydef}

First we need the notion of an atomic module.  We will provide several reasonable definitions, and describe the relations between them. For an idempotent semiring $R$ , we denote $1_R=1$ and $0_R=0$ whenever the context is clear. Also, recall that for any module over an idempotent semiring $R$ is equipped with a canonical partial order, namely, $x\leq y$ if and only if $x+y=y$.

\begin{mydef}\label{definition: weakly cover}
Let $R$ be an idempotent semiring and let $N$ be an $R$-module and $x, y\in N$.  
\begin{enumerate}
    \item 
If $x\leq y$, we let $[x, y] := \{z\in N\mid x\leq z\leq y\}$.
    \item 
We say $y$ \emph{weakly covers} $x$ (denoted $x\preceq y$) if $x\leq y$ and every element of $[x, y] = [x, x + y]$ has the form $x + ay$ where $a\in R$ satisfies $a\leq 1$.
\item 
We say $x$ is a \emph{weak atom} if it weakly covers $0$.
\end{enumerate}   
\end{mydef}

\begin{rmk}
We can apply Definition \ref{definition: weakly cover} to $\mathbb{B}$-modules.  In this case $y$ weakly covers $x$ if and only if either $y=x$ or $y$ covers $x$ in the lattice theoretic sense.  Similarly a weak atom is an element that is either $0$ or an atom in the lattice-theoretic sense.
\end{rmk}

We may fix this minor discrepancy with the classical case by considering a stronger definition instead.  It will help to consider the following maps.

\begin{mydef}
Let $R = K$ or $\mathcal{O}_K$ for an idempotent semifield $K$.  Let $N$ be an $R$-module and let $x, y\in N$.  If $x\leq y$, we use $\phi_{x, y}: \mathcal{O}_K\rightarrow N$ to denote the map $\phi_{x, y}(a) = x + ay$.
\end{mydef}
Note that for $x\leq y$, $y$ weakly covers $x$ if and only if the map $\phi_{x, y}$ is surjective onto $[x,y]$.  We could instead require bijectivity.

\begin{mydef}
Let $R = K$ or $\mathcal{O}_K$ for an idempotent semifield $K$.  Let $N$ be an $R$-module and let $x, y\in N$.  We say $y$ \emph{covers} $x$ (denoted $x\prec y$) if $x\leq y$ and $\phi_{x,y}$ is a bijection.  $x$ is a \emph{atom} if it covers $0$.  We say $N$ is \emph{atomic} if every element is a sum of atoms. 
\end{mydef}

\begin{lem}\label{lemma: multiple of atom}
Let $R = K$ or $\mathcal{O}_K$ for an idempotent semifield $K$. Let $N$ be an $R$-module and $x\in N$ be an atom. Let $a\in R$ be nonzero.  Then $ax$ is an atom. 
\end{lem}
\begin{proof}
    We know either that $a\leq 1$ (if $R = \mathcal{O}_K$) or $a \in R^\times$ (if $R = K$).  Suppose first that $a\leq 1$.  Let $y\leq ax$.  Then $y \leq x$.  By atomicity, there is a unique $b\leq 1$ such that $y = bx$.  Thus $bx \leq ax$, i.e. $ax = (a + b)x$.  By atomicity of $x$, $a = a + b$, i.e. $b\leq a$.  Let $c = a^{-1}b\leq 1$.  Then $y = c(ax)$.  Moreover, this last equation uniquely specifies $c$, since if $y = d(ax)$ then by atomicity of $x$ we have $ac = ad$ and hence $c = d$.

    Now suppose $a$ is a unit and $y\leq ax$.  Then $a^{-1}y\leq x$, so there is a unique $b$ such that $a^{-1}y = bx$ or equivalently $y = b(ax)$.  Hence $ax$ is an atom.
\end{proof}

The notion of quasi-free module was defined in \cite{jun2024tropical} to study representation theory of finite groups over $\mathbb{R}_{\max}$ in connection to valuated matroids. Recall that a finitely generated $\mathbb{B}$-module is quasi-free if and only if it is atomic as a lattice (\cite[Lemma 3.14]{jun2024tropical}).

\begin{mydef}\cite[Definition 3.11]{jun2024tropical}
Let $R$ be an idempotent semiring.  Let $N$ be an $R$-module.  A subset $S\subseteq N$ is \emph{quasi-independent} if and only if for any $x_1, \ldots, x_n\in S$ and any $\{c_ij\mid i,j\leq n\}\in R$ satisfying
\begin{equation}
    x_i = \sum c_{ij} x_j,
\end{equation}
we have $c_{ij} = \delta_{ij}$ for all $i, j$.  A quasi-independent generating set is called a \emph{quasi-basis}.  $N$ is \emph{quasi-free} if it has a quasi-basis.
\end{mydef}

\begin{lem}\label{lemma: Semifield of fractions of valuation semiring}
Let $K$ be an idempotent semifield.  If $x\in K$, there exist $a, b\in \mathcal{O}_K$ with $a = bx$.
\end{lem}
\begin{proof}
Set $b = (x + 1)^{-1}$ and $a = bx$.  Then $b\leq 1$, so $b\in \mathcal{O}_K$.  Also, multiplying $x\leq x + 1$ by $b$ gives $a\leq 1$.
\end{proof}

We have seen several approaches that can be taken to define an atomic $R$-module.  The following proposition relates them.

\begin{pro}\label{proposition: approaches to atomicity}
Let $R$ be an idempotent semiring and $N$ be an $R$-module.
\begin{enumerate}
    \item 
    If $N$ is torsion-free then every weak atom is either $0$ or an atom.  In particular a torsion-free module is atomic if and only if every element is a sum of weak atoms.
 \item 
  Suppose $R = \mathcal{O}_K$ for a totally ordered idempotent semifield $K$ or $R = K$ for an idempotent semifield $K$, and $N$ is finitely generated. If $N$ is atomic, then $N$ is quasi-free. In this case a quasi-basis is the same as a minimal set of atoms that generates $N$. 
    \item  
    Suppose $R = \mathcal{O}_K$ or $R = K$ for an idempotent semifield $K$, and $N$ is finitely generated.  If $N$ is quasi-free then $N$ is atomic.  
   
\end{enumerate}
\end{pro}
\begin{proof}
(1): The torsion free condition implies $\phi_{0, x}$ is injective for all $x\neq 0$.  Thus it is surjective if and only if it is bijective, so $x\neq 0$ is a weak atom if and only if $x$ is an atom.  This implies the first claim.

(2): Suppose $N$ is atomic.  Since $N$ is finitely generated, by writing each generator as a sum of atoms, $N$ is generated by finitely many atoms.  Fix a minimal generating set of atoms $x_1, \ldots, x_n$.  Choose elements $c_{ij}\in R$ such that $x_i = \sum_j c_{ij}x_j$ for all $j$.  Because $x_i$ is an atom, each term is a multiple of $x_i$.  If $c_{ij}\neq 0$ for $j\neq i$ there is some $a$ such that $c_{ij}x_j = ax_i$. If $R$ is a semifield, $c_{ij}^{-1}a\in R$ so $x_j = c_{ij}^{-1}ax_i$.  Then $x_j$  can be dropped from the generating set, contradicting minimality.  On the other hand, if $R = \mathcal{O}_K$, then either $c_{ij}^{-1}a$ or $a^{-1}c_{ij}$ belongs to $R$ depending on whether $c_{ij}\leq a$; either, we may proceed similarly to the semifield case. So for both the semifield case and the totally ordered case, $c_{ij} = 0$ for $i\neq j$ and $x_i = c_{ii} x_i$.  Using Lemma \ref{lemma: Semifield of fractions of valuation semiring}, we may write $c_{ii} = \alpha^{-1}\beta$ with $\alpha, \beta \leq 1$ (we choose $\alpha = 1$ if $R = \mathcal{O}_K$).  From $\alpha x_i = \beta x_i$, and the fact that $x_i$ is an atom, we conclude that $\alpha = \beta$ and hence $c_{ii} = 1$.  Thus the $x_i$ form a quasi-basis and $N$ is quasi-free.

(3): Let $N$ be quasi-free (in addition to being finitely generated).  Fix a quasi-basis $S$.  First we observe that $S$ is a minimal generating set (if $x\in S$ is spanned by other elements of $S$, add $x = x$ to both sides of this relation and apply quasi-independence).  In particular, $S$ is finite because we can write any other finite generating set in terms of finitely many elements of $S$.  Write $S = \{x_1, \ldots, x_n\}$. 

To show that each $x_k$ is an atom, we first show that each $x_k$ is a weak atom. Suppose $y\leq x_k$.  There exist coefficients $a_i$ such that $y = \sum a_i x_i$.  Then $x_k = x_k + \sum a_i x_i$.  By quasi-independence, $a_i = 0$ for $i\neq k$ and $a_k \leq 1$, establishing that $x_k$ is a weak atom.  To see that $x_k$ is an atom observe that quasi-independence prohibits relations of the form $x_k = ax_k$ with $a \neq 1$.  Thus each element of the quasi-basis is an atom, and since the quasi-basis is a minimal generating set, it is a minimal generating set of atoms.  Each element of $N$ can be written $\sum a_i x_i$ where the $x_i$ are atoms, and we may drop terms to assume $a_i \neq 0$.  By Lemma \ref{lemma: multiple of atom}, each $a_i x_i$ is an atom, showing that $N$ is atomic.  
\end{proof}

\begin{lem}\label{lemma: finitely many atoms}
Let $K$ be an idempotent semifield and $N$ be an atomic $R$-module where $R = K$ or $\mathcal{O}_K$.
\begin{enumerate}
    \item 
    $N$ is finitely generated if and only if $N$ is generated by a finite collection of atoms.
    \item 
    Suppose $R = K$.  Consider two atoms $x,y\in N$ to be equivalent if there exists $a\in K^\times$ such that $x = ay$.  Then $N$ is finitely generated if and only if the set of equivalence classes of atoms is finite.
\end{enumerate}
\end{lem}
\begin{proof}
 (1): One direction is trivial, so suppose $N$ is finitely generated, and let $x_1, \ldots, x_n$ be generators.  Write each as a sum of atoms $x_i = y_{i1} + \ldots + y_{im_i}$.  Then any element can be written in the form
    \begin{equation}
        \sum_{i\leq n} c_i x_i = \sum_{i\leq n}\sum_{j\leq m_i} c_i y_{ij}.
    \end{equation}
    Hence the $y_{ij}$ generate $N$, so $N$ is generated by finitely many atoms.

(2): For the second claim, let $R = K$ and suppose the set of equivalence classes of atoms is finite.  Let $x_1, \ldots x_n$ be representatives of each class, so every atom has the form $ax_i$ for some $i$ and some $a\in K^\times$.  Every element is a sum of atoms, so can be written in the form $\sum_k a_k x_{i_k}$.  Hence $x_1, \ldots, x_n$ form a generating set.

Conversely, suppose $N$ is finitely generated, and choose a finite set $\{x_1, \ldots, x_n\}$ of atoms which generate $N$.  Let $y$ be an arbitrary atom.  Write $y = \sum_i c_i x_i$.  Since $y\neq 0$, there is $j$ with $c_jx_j\neq 0$.  Also $c_j x_j \leq y$ so there is some $a_j\in K^\times$ such that $c_j x_j = a_j y$.  Hence $y = a_j^{-1}c_j x_j$.  Hence every atom is equivalent to some $x_j$.
\end{proof}

To define geometric $\mathbb{R}_{\max}$-modules, we need a strengthening of the torsion-free condition.

\begin{rmk}\label{remark: wrong defintion of strongly torsionfree}
A seemingly natural condition to impose on a module is that if $x\leq y$ and $x\neq y$ then $\phi_{x, y}$ is injective, or equivalently that if $x\leq y$, $a, b\leq 1$, and $x + ay = x + by$ then either $x = y$ or $a = b$.  This is too strong; in fact if $K\neq\mathbb{B}$ then the only $K$-module satisfying this property is the zero module, as can be seen by choosing any $x\neq 0$, $b < 1$, $a = 0$ and $y = b^{-1}x$.
\end{rmk}

We instead consider a weakening of the property described in the above remark, which avoids this sort of counterexample.  Since this weakening no longer implies the torsion-free condition (in the case of $\mathcal{O}_K$-modules), we augment it by a weakening of the torsion-free condition.

\begin{mydef}\label{definition: strongly torsionfree}
Let $R = K$ or $\mathcal{O}_K$ for an idempotent semifield $K$.  Let $N$ be an $R$-module.  We say $N$ is \emph{strongly torsion-free} if 
\begin{enumerate}
    \item 
    If $x\in N\backslash\{0\}$ and $a\in R\backslash\{0\}$ then $ax\neq 0$.
    \item 
    Suppose $x, y\in N$ with $x\leq y$.  Let $a, b \in \mathcal{O}_K$.  Suppose $x + ay = x + by$.  Then either $a = b$ or $x = x + ay = x + by$.
\end{enumerate}
\end{mydef}
Note that every $\mathbb{B}$-module is strongly torsion-free. This explains why this condition does not appear in the definition of a geometric lattice.

For an $\mathcal{O}_K$-module, the strong property described in Remark \ref{remark: wrong defintion of strongly torsionfree} essentially says $N / \langle x\rangle$ is always torsion-free.  To instead describe the strongly torsion-free condition in terms of quotients, it is helpful to decompose the torsion-free condition into two conditions.

\begin{lem}\label{lemma: decomposition of torsion-free}
Let $R = K$ or $\mathcal{O}_K$ for an idempotent semifield $K$. Let $N$ be an $R$-module. Then $N$ is torsion-free if and only if the following hold:
\begin{enumerate}
    \item
If $x\in N\backslash\{0\}$ and $a\in R\backslash\{0\}$ then $ax\neq 0$.
    \item 
Let $x\in N$ and $a,b\in R$ and suppose $ax = bx$.  Then $a = b$ or $ax = bx = 0$.
\end{enumerate}
Additionally, Condition (2) is equivalent to the statement that if $x\in N$, $a, b\in\mathcal{O}_K$ and $ax=bx$ then $a=b$ or $ax=bx=0$.  
\end{lem}
\begin{proof}
Suppose $N$ satisfies (1) and (2), and suppose $ax = bx$.  If $ax = bx = 0$ then $x = 0$ or $a = 0 = b$.  Otherwise $a = b$.  Thus $x = 0 $ or $a = b$, so $M$ is torsion-free.

Conversely, suppose $N$ is torsion-free.  If $ax = 0 = 0x$ then $x = 0$ or $a = 0$, which proves (1).  Now suppose $ax = bx$.  Then $a = b$ or $x = 0$ and hence $ax = bx = 0$.  This proves (2).

For the last claim, there is nothing to prove if $R = \mathcal{O}_K$, so assume $R = K$. One direction is trivial since $\mathcal{O}_K\subseteq K$. Therefore, suppose that if $x\in N$ and $a, b\in\mathcal{O}_K$ with $ax=bx$ then $a=b$.  Let $x\in N$ and $a, b\in K$ satisfy $ax = bx$.  Divide by $a\vee b$ to get $\frac{a}{a\vee b}x = \frac{b}{a\vee b}x$.  Since the coefficients now lie in $\mathcal{O}_K$, we either have $\frac{a}{a\vee b} = \frac{b}{a\vee b}$ (in which case $a = b$) or $\frac{a}{a\vee b}x = \frac{b}{a\vee b}x = 0$ (in which case $ax=bx=0$).
\end{proof}

\begin{mydef}
Let $R$ be an idempotent semiring. 
\begin{enumerate}
    \item 
An $R$-module $L$ is \emph{subtractive} if $a+b, b \in L$ implies $a \in L$ for any $a,b \in L$.
\item 
For an $R$-module $N$, by a finitely generated subtractive submodule $L$, we mean the smallest subtractive submodule of $N$ containing a given finite set.\footnote{This doesn't imply that $L$ is finitely generated as a submodule.}
\end{enumerate}
\end{mydef}

\begin{lem}\label{lem: subtractive}
Let $K$ be an idempotent semifield and $N$ be an $\mathcal{O}_K$-module. The finitely generated subtractive submodules of $N$ are precisely the intervals $[0,x]$ for $x\in N$.    
\end{lem}
\begin{proof}
These intervals are closed under addition by idempotence, and are closed under scalar multiplication since all elements of $\mathcal{O}_K$ are at most $1$.  Furthermore for any $x_1, \ldots, x_n\in N$ and any subtractive submodule $L\subseteq N$ containing the $x_i$, we have $x_i\in [0, x_1 + \ldots + x_n] \subseteq L$ for all $i$.  This proves the lemma.  
\end{proof}

\begin{lem}
Let $K$ be an idempotent semifield and $N$ be an $\mathcal{O}_K$-module.  The following are equivalent
\begin{enumerate}
\item 
$N$ is strongly torsion-free.
\item 
$N$ satisfies Condition (1) of Lemma \ref{lemma: decomposition of torsion-free} and for any subtractive submodule $L\subseteq N$, $N / L$ satisfies Condition (2) of Lemma \ref{lemma: decomposition of torsion-free}.
\item 
$N$ satisfies Condition (1) of Lemma \ref{lemma: decomposition of torsion-free}  and for any finitely generated subtractive submodule $L\subseteq N$, $N / L$ satisfies Condition (2) of Lemma \ref{lemma: decomposition of torsion-free}.
\end{enumerate}
\end{lem}
\begin{proof} 
To show (1) implies (3), suppose $N$ is strongly torsion-free. By Lemma \ref{lem: subtractive}, we may assume that $L=[0,x]$ for some $x \in N$. By definition $N$ satisfies Condition (1) of Lemma \ref{lemma: decomposition of torsion-free}.  We aim to show $N / [0, x]$ satisfies Condition (2) of Lemma \ref{lemma: decomposition of torsion-free}.  Suppose $a\bar{y} = b\bar{y}$ for $a, b\in \mathcal{O}_K$ and $\bar{y}\in N/[0,x]$.  Adding $\bar{x}$ to both sides, we obtain $\overline{\phi_{x, y}(a)} = \overline{\phi_{x, y}(b)}$.  Thus there exist $u, v\in [0, x]$ such that
\begin{equation}
\phi_{x, y}(a) = u + \phi_{x, y}(a) = v + \phi_{x, y}(b) = \phi_{x, y}(b)
\end{equation}
where the first equality uses $\phi_{x, y}(a) \geq x \geq u$.  Since $N$ is strongly torsion-free, this implies $a = b$ unless $x = x + ay = x + by$, in which case $a\bar{y} = b\bar{y} = 0$.

Conversely, suppose $N / [0, x]$ satisfies Condition (2) of Lemma \ref{lemma: decomposition of torsion-free} for all $x\in N$ and $N$ satisfies Condition (1) of Lemma \ref{lemma: decomposition of torsion-free}.  Let $x\leq y$ and let $a,b\in\mathcal{O}_K$ be such that $x + ay = x + by$.  In $N / [0, x]$, we have $a\bar{y} = b\bar{y}$. By Condition (2) of Lemma \ref{lemma: decomposition of torsion-free}, either $a\bar{y} = b\bar{y} = 0$ (in which case $x = x + ay = x + by$) or $a = b$.  In combination with Condition (1) of Lemma \ref{lemma: decomposition of torsion-free}, this says $N$ is strongly torsion-free

(2) trivially implies (3), so it remains to show (3) implies (2).  Let $L\subseteq N$ be a subtractive submodule.  Write it as a filtered union $\bigcup L_i$ of finitely generated subtractive submodules. Let $a,b\in \mathcal{O}_K$ be such that $a\bar{y} = b\bar{y}$, i.e., there exist $u,v\in L$ with $ay + u = by + v$.  There is some $i$ such that $u, v\in L_i$, and reducing modulo this $L_i$ yields $a\bar{y}_{L_i} = b\bar{y}_{L_i}$.  Since $N/ L_i$ satisfies Condition (2) of Lemma \ref{lemma: decomposition of torsion-free}, either $a = b$ or $a\bar{y}_{L_i} = b\bar{y}_{L_i} = 0$, in which case $a\bar{y} = b\bar{y} = 0$.  Then $N/L$ satisfies Condition (2) of Lemma \ref{lemma: decomposition of torsion-free}.
\end{proof}

In particular, by taking $L = 0$, we see that strongly torsion-free $\mathcal{O}_K$-modules are torsion-free, which justifies the terminology.  It is not hard to check that strongly torsion-free $K$-modules are also torsion-free - one approach is the following lemma.

\begin{lem}\label{lemma: torision-free lemma}
Let $N$ be a $K$-module where $K$ is an idempotent semifield.  Then
\begin{enumerate}
\item 
$N$ is a torsion-free $K$-module if and only if $N$ is a torsion-free $\mathcal{O}_K$-module.
\item 
$N$ is a strongly torsion-free $K$-module if and only if $N$ is a strongly torsion-free $\mathcal{O}_K$-module.
\end{enumerate}
\end{lem}
\begin{proof}
Condition (2) of Lemma \ref{lemma: decomposition of torsion-free} or Definition \ref{definition: strongly torsionfree} only involves the $\mathcal{O}_K$-module structure of $N$.  

Thus we only need to consider their first conditions (which are identical).  For this we need the fact that $K^\times$ has greatest lower bounds; this fact is a result of the observation that inversion is order reversing so $x \wedge y = (x^{-1} \vee y^{-1})^{-1}$.

Suppose $N$ satisfies the first condition as an $\mathcal{O}_K$-module, i.e. if $ax = 0$ with $a \in \mathcal{O}_K\backslash\{0\}$ then $x = 0$.  Suppose that $ax = 0$ for some $a\in K^\times$.  Then $a \wedge 1 \in \mathcal{O}_K\backslash\{0\}$, and $(a\wedge 1)x\leq ax = 0$.  Thus $x = 0$.  The converse direction is trivial since $\mathcal{O}_K\subseteq K$.
\end{proof}

It is known that a lattice is semimodular if and only if $x \preceq y$ implies that $x + z\preceq y + z$ for all $z$. For instance, see \cite{stanley1974finite}. 

\begin{mydef}
Let $R = K$ or $R=\mathcal{O}_K$ for an idempotent semifield $K$ and let $N$ be an $R$-module.  
\begin{enumerate}
    \item 
$N$ is \emph{semimodular} if for any $x, y, z\in N$ with $x\preceq y$ we have $x + z\preceq y + z$.\footnote{Note that $x \preceq y$ means that $y$ weakly covers $x$ as in Definition \ref{definition: weakly cover}.}
    \item 
$N$ is \emph{geometric} if $N$ is strongly torsion-free, atomic, and semimodular.    
\end{enumerate}
\end{mydef}

\begin{lem}\label{lemma: covering semimodulear}
Let $M$ be a valuated matroid on a finite set $E$, and consider the covering relation $\preceq$ in the join semilattice of finitely generated flats.  If $\mathrm{cl}(S)\preceq\mathrm{cl}(T)$ (where $S, T\subseteq \hat{M}$ are finite) then either $\mathrm{cl}(S) = \mathrm{cl}(T)$ or there is some $cx\in T\backslash S$ such that $\mathrm{cl}(T) = \mathrm{cl}(S \cup \{cx\})$.
\end{lem}
\begin{proof}
Without loss of generality, we may suppose $S\subseteq T$ as we may replace $T$ with $S\cup T$.  Pick a minimal counterexample for which $S\subseteq T$.  

\underline{Claim}: First we claim that if $ay\in T\backslash S$ then $\mathrm{cl}(T)\neq \mathrm{cl}(T \backslash \{ay\})$.  Suppose otherwise.  If $\mathrm{cl}(S)\preceq \mathrm{cl}(T) = \mathrm{cl}(T\backslash\{ay\})$ then by minimality of $T$, either $\mathrm{cl}(S) = \mathrm{cl}(T\backslash\{ay\}) = \mathrm{cl}(T)$ or there is $cx\in (T\backslash\{ay\})\backslash S\subseteq T\backslash S$ such that $\mathrm{cl}(T) = \mathrm{cl}(T\backslash\{ay\}) = \mathrm{cl}(S\cup\{cx\})$.  In either case, we contradict that $T$ is a counterexample.  Thus the claim holds. 

In the following we prove that $T\backslash S$ does not have two distinct elements. 

\underline{Case 1}: Suppose that $ax, bx\in T\backslash S$ for some $a,b\in \mathbb{R}_{\max}$ (with $a \neq b$) and $x\in E$. We may suppose without loss of generality that $a\leq b$ and then $\mathrm{cl}(T) = \mathrm{cl}(T \backslash \{ax\})$ by Condition (3) of Proposition \ref{proposition: valuated closure axioms}. Since $ax\in T\backslash S$, this contradicts the above claim.

\underline{Case 2}: Suppose that $ax, by\in T\backslash S$ are distinct.  We may assume $x\neq y$ as we have already dealt with the $x = y$ case.  Since $\mathrm{cl}(S)\preceq\mathrm{cl}(T)$ and $\mathrm{cl}(S)\subseteq\mathrm{cl}(S\cup \{ax\}) \subseteq\mathrm{cl}(T)$, by the definition of weak cover (applied to $\text{cl}(S)\preceq \text{cl}(T)$), there is some $r\leq 1$ such that $\mathrm{cl}(S \cup \{ax\}) = \mathrm{cl}(S\cup rT)$, and in particular, $rby\in \mathrm{cl}(S \cup \{ax\})$ and $ax\in \mathrm{cl}(S \cup (rT \backslash \{rax\}) \cup \{rax\})$.  Suppose $r < 1$.  By Conditions (3) and (5) of Proposition \ref{proposition: valuated closure axioms}, $ax\in \mathrm{cl}(S \cup (rT \backslash \{rax\}))\subseteq \mathrm{cl}(S \cup T \backslash \{ax\}) = \mathrm{cl}(T\backslash\{ax\})$.   This contradicts the above claim, so we must have $r = 1$ instead.  But then $by\in \mathrm{cl}(S \cup \{ax\})\subseteq\mathrm{cl}(T\backslash\{by\})$ so we may replace $T$ with $T\backslash\{by\}$ which also contradicts the above claim. 

Hence, we cannot have two distinct elements $ax, by\in T\backslash S$.  If $T\backslash S$ has a single element, denote it by $cx\in T\backslash S$.  Then $T = S \cup \{cx\}$.  Hence $\mathrm{cl}(T) = \mathrm{cl}(S\cup \{cx\})$ contradicting that $T$ is a counterexample.  Hence $T\backslash S = \emptyset$ and $T = S$.  But in this case, $\mathrm{cl}(T) = \mathrm{cl}(S)$, again contradicting that $T$ is a counterexample.
\end{proof}

In the next proposition, we prove that $Q_M$ (as an $\mathbb{R}_{\max}$-module) is finitely generated and geometric. In doing so, we view $Q_M$ as the $\mathbb{R}_{\max}$-module of finitely generated flats of a valuated matroid $M$ through Proposition \ref{proposition: iso $Q_M$ to flats}. In particular, we use $\vee$ for addition of $Q_M$. For instance $\text{cl}(S),\text{cl}(T) \in Q_M$, and $\text{cl}(S)+\text{cl}(T) := \text{cl}(S)\vee \text{cl}(T)=\text{cl}(S\cup T)$.

\begin{pro}\label{proposition: $Q_M$ geometric}
Let $M$ be a valuated matroid on a finite set $E$.  Then the $\mathbb{R}_{\max}$-module $Q_M$ of finitely generated flats is finitely generated and geometric.
\end{pro}
\begin{proof}
\underline{Strong torsion-freeness}: First we check that $Q_M$ is strongly torsion-free.  Condition (1) of Definition \ref{definition: strongly torsionfree} holds for all modules over a semifield: if $ax = 0$ and $a\neq 0$ then $x = a^{-1}ax = 0$.  Thus we check Condition (2) of Definition \ref{definition: strongly torsionfree}.

Let $a, b\in\mathbb{R}_{\max}$ with $a,b\leq 1$ and $\mathrm{cl}(S), \mathrm{cl}(T)$ be finitely generated flats (where $S, T\subseteq\hat{M}$ are finite) such that $\mathrm{cl}(S)\subsetneq \mathrm{cl}(T)$, and 
\begin{equation}
\mathrm{cl}(S) \vee a\mathrm{cl}(T) = \mathrm{cl}(S) \vee b\mathrm{cl}(T), \quad \textrm{that is,} \quad \mathrm{cl}(S\cup aT) = \mathrm{cl}(S \cup bT).
\end{equation}
We may assume $S\subseteq T$ because replacing $T$ with $S\cup T$ does not change $\mathrm{cl}(T)$. We wish to show that either $a = b$ or $\mathrm{cl}(S) = \mathrm{cl}(S\cup bT)$.  Assume $a \neq b$, and without loss of generality assume $a < b$.

Now assume $T$ is a minimal counterexample with the stated properties:
\begin{equation}\label{eq: eq min}
S\subseteq T, \quad \mathrm{cl}(S)\subsetneq \mathrm{cl}(T), \quad \mathrm{cl}(S) \neq \mathrm{cl}(S \cup bT), \quad \textrm{and} \quad \mathrm{cl}(S) \vee a\mathrm{cl}(T) = \mathrm{cl}(S) \vee b\mathrm{cl}(T).
\end{equation}
Let $cx\in T \backslash\mathrm{cl}(S)$ (which exists as otherwise $\mathrm{cl}(T)\subseteq\mathrm{cl}(S)$).

Observe
\begin{equation}
    bcx \in \mathrm{cl}(S \cup bT) = \mathrm{cl}(S \cup aT) = \mathrm{cl}(S \cup aT' \cup \{acx\}),
\end{equation}
where $T'=T\backslash \{cx\}$. By Conditions (3) and (5) of Proposition \ref{proposition: valuated closure axioms}, we have
\begin{equation}
 bcx\in \mathrm{cl}(S \cup aT') \subseteq \mathrm{cl}(S\cup bT').  \end{equation}
By multiplying by $b^{-1}a\leq 1$, we get 
\begin{equation}
    acx \in \mathrm{cl}(b^{-1}a(S\cup aT')) \subseteq \mathrm{cl}(S \cup aT').
\end{equation}
Thus $S \cup aT \subseteq \mathrm{cl}(S\cup aT')$, implying that $\mathrm{cl}(S \cup aT) = \mathrm{cl}(S \cup aT')$. Therefore, we have
\begin{equation}\label{eq: eq min2}
    \mathrm{cl}(S \cup bT) = \mathrm{cl}(S \cup aT) = \mathrm{cl}(S \cup aT') \subseteq \mathrm{cl}(S \cup bT') \subseteq\mathrm{cl}(S \cup bT).
\end{equation}
Hence we have equality throughout. 
So, we have
\[
S\subseteq T', \quad \mathrm{cl}(S)\subseteq \mathrm{cl}(T'), \quad \textrm{and} \quad \mathrm{cl}(S) \vee a\mathrm{cl}(T') = \mathrm{cl}(S) \vee b\mathrm{cl}(T').
\]
By minimality of $T$ in \eqref{eq: eq min}, we should have either $\mathrm{cl}(S)=\mathrm{cl}(T')$ or $\mathrm{cl}(S)=\mathrm{cl}(S \cup bT')$. But, if $\mathrm{cl}(S)=\mathrm{cl}(T')$, since $b\leq 1$, we also have
\begin{equation}
   \mathrm{cl}(S \cup bT) = \mathrm{cl}(S \cup bT') = \mathrm{cl}(S) \vee b\mathrm{cl}(T') = \mathrm{cl}(S) \vee b\mathrm{cl}(S) = \mathrm{cl}(S).
\end{equation}
So, in both cases, we have contradiction. This shows $Q_M$ is strongly torsion-free.

\underline{Atomicity and Finite generation}: Next we check that $Q_M$ is atomic.  Let $x\in E$.  We claim that $F :=\mathrm{cl}(\{x\})$ is a weak atom.  To see this, suppose $\mathrm{cl}(S)\subseteq F$.  We wish to show there is some $c\leq 1$ such that $\mathrm{cl}(S) = cF$.  We proceed by induction on the number of elements $ay\in S$ with $y\neq x$.  For the case where there are no such elements, say $S = \{a_1x, \ldots, a_n x\}$, then Conditions (2) and (3) of Proposition \ref{proposition: valuated closure axioms} imply $\mathrm{cl}(S) = (\max_i a_i) F$, which establishes the claim except when $\max_i a_i > 1$.  In this case, Condition (3) of Proposition \ref{proposition: valuated closure axioms} gives 
\begin{equation}
F\subseteq (\max_i a_i)F = \mathrm{cl}(S)\subseteq F.    
\end{equation}
This establishes the base case.

Now let $ay\in S$ with $y\neq x$ and let $S' = S \backslash\{ay\}$.  By Condition (4) of Proposition \ref{proposition: valuated closure axioms} (applied to the empty set), there is some $b\leq 1$ such that
\begin{equation}
   ay \in \mathrm{cl}(\{bx\}) \quad \textrm{and} \quad  bx\in \mathrm{cl}(\{ay\}).
\end{equation}
Hence $\mathrm{cl}(S) = \mathrm{cl}(S' \cup \{bx\})$ (each is contained in the other since both contain $ay$ and $bx$ and $S'$).  Moreover, the inductive hypothesis implies there is $c\leq 1$ such that
\begin{equation}
 \mathrm{cl}(S) = \mathrm{cl}(S' \cup \{bx\}) = cF.   
\end{equation}
This establishes the claim.  Because $Q_M$ is torsion-free, $\mathrm{cl}(\{x\})$ is an atom.  By Lemma \ref{lemma: multiple of atom}, $\mathrm{cl}(\{dx\})$ is an atom for any $d\in \mathbb{R}_{\max}$.  Since such elements generate $Q_M$ as a semigroup, $Q_M$ is atomic.

Our description of the atoms together with the finiteness of $E$ and Lemma \ref{lemma: finitely many atoms}, implies $M$ is finitely generated.

\underline{Semimodularity}:  Finally, we must show semimodularity.  Let $S, T, U\subseteq \hat{M}$ be finite and satisfy $\mathrm{cl}(S)\preceq\mathrm{cl}(T)$.  We wish to show $\mathrm{cl}(S\cup U)\preceq\mathrm{cl}(T\cup U)$.  Without loss of generality, we may assume that $S\subseteq T$ and $S\subseteq U$.  We also assume $S \neq T$, since the result is trivial otherwise. Also, we may assume that $\text{cl}(S) \neq \text{cl}(T)$. By Lemma \ref{lemma: covering semimodulear}, there is some $cx\in T\backslash S$ such that $\mathrm{cl}(T) = \mathrm{cl}(S\cup\{cx\})$, and hence we have 
\begin{equation}
\mathrm{cl}(T\cup U) = \mathrm{cl}(\mathrm{cl}(T)\cup U) = \mathrm{cl}(\mathrm{cl}(S\cup\{cx\})\cup U) = \mathrm{cl}(S\cup \{cx\} \cup U) = \mathrm{cl}(U\cup \{cx\}).
\end{equation}

Semimodularity then reduces to showing that if $S\subseteq U$ and $\mathrm{cl}(S) \preceq \mathrm{cl}(S \cup \{cx\})$ then 
\begin{equation}\label{eq: assume}
\mathrm{cl}(U) \preceq \mathrm{cl}(U \cup \{cx\}).    
\end{equation} 
Let $F$ be a finitely generated flat such that $\mathrm{cl}(U)\subseteq F\subseteq\mathrm{cl}(U\cup\{cx\})$.  We must show there is some $b\leq c$ such that $F = \mathrm{cl}(U \cup \{bx\})$.  We may suppose $F\neq \mathrm{cl}(U)$, because otherwise $b=0$ works.  

Let $V$ minimize the number of supports of elements that are not multiples of $x$ among all subsets with $U\subseteq V$ and $\mathrm{cl}(V) = F$.  Suppose $ay\in V\backslash U$ with $y\neq x$.  Then $ay\in\mathrm{cl}(U\cup\{cx\})$.  By Condition (4) of Proposition \ref{proposition: valuated closure axioms}, there exists $d\leq c$ such that
\begin{equation}
 dx\in \mathrm{cl}(U\cup\{ay\}) \quad \textrm{and} \quad ay\in\mathrm{cl}(U\cup\{dx\}).
\end{equation}
Thus $\mathrm{cl}(U \cup \{ay\}) = \mathrm{cl}(U\cup\{dx\})$ and so $\mathrm{cl}(V) = \mathrm{cl}(V \cup \{dx\} \backslash \{ay\})$.  This contradicts minimality of $V$.  Hence all elements of $V\backslash U$ are multiples of $x$, say $V = U \cup \{b_1x, \ldots, b_n x\}$.  Letting $b = \max b_i$, it is easy to see that $\mathrm{cl}(V) = \mathrm{cl}(U \cup \{bx\})$, contradicting minimality unless $V = U \cup \{bx\}$.  Thus $F = \mathrm{cl}(U\cup\{bx\})$.  This proves semimodularity if $b \leq c$.

If $b > c$, then with \eqref{eq: assume} we have
\begin{equation}
    \mathrm{cl}(U \cup \{bx\}) = F \subseteq \mathrm{cl}(U \cup \{cx\}) \subseteq \mathrm{cl}(U \cup \{bx\}).
\end{equation}
Thus we have equality throughout and $F = \mathrm{cl}(U \cup \{cx\})$ (and $c\leq c$).  This proves semimodularity.
\end{proof}

The geometric $\mathbb{R}_{\max}$-module $Q_M$ only determines $M$ up to projective equivalence.  To recover the projective equivalence class of $M$, it will help to relate projective equivalence to our closure operator.

\begin{lem}\label{lemma: iso val mat}
    Let $M, M'$ be valuated matroids on finite sets $E, E'$.  Then
    \begin{enumerate}
        \item There is a one-to-one correspondence between bijections $\psi: \hat{M}\rightarrow\hat{M'}$ which are compatible with the $\mathbb{R}_{\max}$-action (i.e. $\psi(ax) = a\psi(x)$ for $a\in\mathbb{R}_{\max}$ and $x\in \hat{M}$) and pairs $(\phi, s)$ where $\phi: E\rightarrow E'$ is a bijection and $s: E\rightarrow \mathbb{R}_{\max}^\times$ is any map.
        \item There is a one-to-one correspondence between bijections $\psi: \hat{M}\rightarrow\hat{M'}$ which are compatible with the $\mathbb{R}_{\max}$ action and which preserve circuit sets (in the sense that $C\subseteq\hat{M}$ is a circuit set if and only if $\psi(C)$ is a circuit set of $\hat{M'}$) and projective equivalences $(\phi, s)$.
        \item There is a one-to-one correspondence between bijections $\psi: \hat{M}\rightarrow\hat{M'}$ which are compatible with the $\mathbb{R}_{\max}$ action and which preserve closure (in the sense that for any $S\subseteq\hat{M}$, $\mathrm{cl}(\psi(S)) = \psi(\mathrm{cl}(S)$) and projective equivalences $(\phi, s)$.
    \end{enumerate}
\end{lem}
\begin{proof}
(1): Choose $(\phi, s)$ so that $\phi$ is a bijection.  For $a\in \mathbb{R}_{\max}$ and $x\in E$, let $\psi(ax) = a s(x)\phi(x)$.  This is clearly compatible with the $\mathbb{R}_{\max}$-action.  To see it is injective, let $a'x'\in \hat{M}$ with $a s(x) \phi(x) = a' s(x')\phi(x')$.  Then $\phi(x) = \phi(x')$ as elements of $E$, so $x = x'$.  Dividing by $s(x)$ gives $a = a'$.  To see $\psi$ is surjective, let $b\in\mathbb{R}_{\max}$ and $y\in E'$.  There is some $x\in E$ with $\phi(x) = y$.  Letting $a = s(x)^{-1}b$, we have $\psi(ax) = by$.

To construct the inverse, we now choose a bijection $\psi: \hat{M}\rightarrow\hat{M'}$ compatible with the $\mathbb{R}_{\max}$-action.  Then $\psi(\mathbf{0}) = \mathbf{0}$, so $\psi(x)\neq \mathbf{0}$ for $x\in E$.  Consequently, for each $x\in E$, there exists $s(x)\in \mathbb{R}_{\max}^\times$ and $\phi(x) \in E'$ such that $\psi(x) = s(x) \phi(x)$.  $\phi$ is bijective because $\psi: \hat{M}\backslash\{\mathbf{0}\}\rightarrow\hat{M'}\backslash\{\mathbf{0}\}$ induces a bijection on $\mathbb{R}_{\max}^\times$ orbits, and these orbits are in bijection with $E$ and $E'$ respectively.  For $a\in \mathbb{R}_{\max}$ we have $\psi(ax) = a s(x) \phi(x)$, and this formula uniquely determines $(\phi, s)$ so our two constructions are mutually inverse.

(2): Suppose that $\psi:\hat{M} \to \hat{M}'$ is a bijection compatible with the $\mathbb{R}_{\max}$ action preserving circuit sets. Let $(\phi,s)$ be the corresponding pair defined in (1). Let $\eta:E \to \mathbb{R}_{\max}$ be a circuit of $M$. We want to show that the following function is a circuit of $M'$:
\begin{equation}\label{eq: circuit check}
\eta':E' \to \mathbb{R}_{\max}, \quad y\mapsto s(\phi^{-1}(y))\eta(\phi^{-1}(y)).
\end{equation}
But, any circuit $\eta:E \to \mathbb{R}_{\max}$ of $M$ uniquely corresponds to the following circuit set:
\[
C_\eta=\{ ax \mid \eta(x)=a\neq -\infty\},  
\]
hence we have
\[
\psi(C_\eta)= \{\psi(ax) \mid ax \in C_\eta\} = \{ as(x)\phi(x) \mid ax\in C_\eta\}=\{ \eta(x)s(x)\phi(x) \mid x \in \mathrm{supp}(\eta)\}
\]
\[
=\{s(\phi^{-1}(y))\eta(\phi^{-1}(y))y \mid y \in \mathrm{supp}(\eta\circ\phi^{-1})\}= \{\eta'(y)y \mid y \in \mathrm{supp}(\eta')\}=C_{\eta'},
\]  
showing that $\eta'$ is a circuit of $M'$. 

Conversely, one can easily see from the above computation that if $(\phi,s)$ preserves circuits then the corresponding $\psi$ preserves circuit sets. 

(3): From (2), it is enough to show that $\psi$ preserves circuit sets if and only if it preserves closure. Note that for any bijection $\psi:\hat{M} \to \hat{M}'$ which is compatible with the $\mathbb{R}_{\max}$ action and $S \subseteq \hat{M}$, we have
\[
\psi(S_\downarrow)=\psi(S)_\downarrow.
\]
From Lemma \ref{lemma: circuits from valuated flats}, for
\begin{equation}
X=\{v\in \hat{M} \mid \exists C\in \mathcal{C} \textrm{ s.t. } v\in C, C\backslash\{v\}\subseteq S_\downarrow \},
\end{equation}
we have
\begin{equation}
    \psi(\mathrm{cl}(S)) = \psi(S)_\downarrow \cup \psi(X). 
\end{equation}
But, since $\psi$ preserves circuit sets, we have
\[
\psi(X) = \{  u \in \hat{M}' \mid \exists C'\in \mathcal{C'} \textrm{ s.t. } u \in C', C'\backslash\{u\}\subseteq \psi(S)_\downarrow\},
\]
where $\mathcal{C}'$ is the set of circuit sets of $M'$. It follows that 
\[
\psi(\text{cl}(S))=\psi(S)_\downarrow \cup \{  u \in \hat{M}' \mid \exists C'\in \mathcal{C'} \textrm{ s.t. } u \in C', C'\backslash\{u\}\subseteq \psi(S)_\downarrow\} = \text{cl}(\psi(S)).
\]
Conversely, from Lemma \ref{lemma: closure operator determines valuated circuits}, it is clear that if $\psi$ preserves closure, then it preserves circuit sets. 
\end{proof}

\begin{rmk}
Lemma \ref{lemma: iso val mat} (2) shows that when we restrict to $\mathbb{B}=\{0,1\}$, this is nothing but isomorphism classes of matroids (with strong maps).
\end{rmk}

\begin{mythm}\label{theorem: geometric modules1}
Let $N$ be a finitely generated geometric $\mathbb{R}_{\max}$-module. Then, there exists a simple valuated matroid $M$ such that $N$ is isomorphic to $Q_M$. 
\end{mythm}
\begin{proof}
Let $N$ be a finitely generated geometric $\mathbb{R}_{\max}$-module. To construct a simple valuated matroid $M$, we first define its underlying set $E$ as follows: we partition the set of atoms of $N$ into orbits under scalar multiplication by $\mathbb{R}_{\max}^\times$ (using Lemma \ref{lemma: multiple of atom}) and let $E$ contain one element of each orbit. Since $N$ is finitely generated, from Lemma \ref{lemma: finitely many atoms}, $E$ is finite. 

Let $\hat{M} = (\mathbb{R}_{\max}\times E) / (\{-\infty\} \times E)$. Since $N$ is strongly torsion-free, from Proposition \ref{proposition: approaches to atomicity}, every weak atom is either $0$ or an atom. Hence, we may identify $\hat{M}$ with the set of weak atoms. For $S\subseteq\hat{M}$ define $\mathrm{cl}$ by
\begin{equation}\label{definition: closure1}
 \mathrm{cl}(S) := \{y\in \hat{M} \mid \exists x_1, \ldots, x_n\in S \textrm{ such that }y \leq \sum_{i=1}^n x_i\}.   
\end{equation}
Clearly $S\subseteq\mathrm{cl}(S)$ and if $S\subseteq T$ then $\mathrm{cl}(S)\subseteq\mathrm{cl}(T)$.  Each element $x\in\mathrm{cl}(\mathrm{cl}(S))$ is bounded by a sum of finitely many elements of $\mathrm{cl}(S)$ each of which is bounded by a finite sum of elements of $S$ and hence $x$ is bounded by a finite sum of elements of $S$ and $x\in\mathrm{cl}(S)$.  Thus $\mathrm{cl}$ is a closure operator.  Clearly any element $y\in\mathrm{cl}(S)$ belongs to $\mathrm{cl}(T)$ for a finite subset $T\subseteq S$.  Hence $\mathrm{cl}$ is finitary. We call $S \subseteq \hat{M}$ closed if $\text{cl}(S)=S$.

Let $Q_M$ be the join semi-lattice of finitely generated closed sets, together with the scalar multiplication:
\begin{equation}
 a\mathrm{cl}(S) := \{ax\mid x\in\mathrm{cl}(S)\}.   
\end{equation}
Now, one can easily see that $a\mathrm{cl}(S)=\mathrm{cl}(aS)$. In fact, it is clear that $a\mathrm{cl}(S)\subseteq \mathrm{cl}(aS)$. For the other inclusion, let $y \in \mathrm{cl}(aS)$. Then, there are $x_i \in aS$ such that $y \leq \sum x_i$, or if we write $x_i=ay_i$ for some $y_i \in S$, we have
\begin{equation}
 y \leq \sum ay_i = a(\sum y_i) \in a\text{cl}(S). 
 \end{equation}

Now, we prove that $Q_M$ is isomorphic to $N$. Define a map $\psi:Q_M \rightarrow N$ by
\begin{equation}\label{eq: map def}
\psi(\mathrm{cl}(S)) = \sum_{x\in S} x \quad \textrm{(for finite $S\subseteq\hat{M}$)}.   
\end{equation}
We claim that this is a well-defined isomorphism of $\mathbb{R}_{\max}$-modules. We first show that $\eqref{eq: map def}$ depends on $S$ only through $\mathrm{cl}(S)$, i.e., if $\text{cl}(S)=\text{cl}(T)$, then $\psi(\mathrm{cl}(S)) =\psi(\mathrm{cl}(T)) $. Indeed, observe that any element of $\mathrm{cl}(S)$ and hence any sum of such elements is bounded by $\psi(\mathrm{cl}(S))$.  On the other hand $\psi(\mathrm{cl}(S))$ is a sum of elements of $\mathrm{cl}(S)$.  Thus we can alternatively describe $\psi(F)$ as the maximum of all sums of elements of $F$.  Note also that
\begin{equation}
 \mathrm{cl}(S) = \{x\in\hat{M}\mid x\leq \psi(\mathrm{cl}(S))\}   
\end{equation}
since if $x\leq \psi(\mathrm{cl}(S))$ then $x$ is bounded by a sum of elements of $S$ and if $x\in\mathrm{cl}(S)$ then $x$ is bounded by the maximum of all sums of elements of $\mathrm{cl}(S)$.

Surjectivity of $\psi$ follows from the assumption that every element is a sum of weak atoms (Proposition \ref{proposition: approaches to atomicity}).  For injectivity, suppose $\psi(\mathrm{cl}(S)) = \psi(\mathrm{cl}(T))$. Then, we have
\begin{equation}
    \mathrm{cl}(S) = \{x\in\hat{M}\mid x\leq \psi(\mathrm{cl}(S))\} = \{x\in\hat{M}\mid x\leq \psi(\mathrm{cl}(T))\} = \mathrm{cl}(T),
\end{equation}
showing that $\psi$ is injective, and hence $\psi$ is bijective.

To see that $\psi$ is a homomorphism, we have
\begin{equation}
 \psi(\mathrm{cl}(S) \vee \mathrm{cl}(T)) = \psi(\mathrm{cl}(S\cup T)) = \psi(\mathrm{cl}(S)) + \psi(\mathrm{cl}(T)).   
\end{equation}
Also for $a\in\mathbb{R}_{\max}$, we have $\psi(a\mathrm{cl}(S)) = \sum_{ax\in aS} ax = a\psi(\mathrm{cl}(S))$. Thus $\psi$ is an isomorphism.  In particular, $Q_M$ is a geometric $\mathbb{R}_{\max}$-module.

Next we prove that the closure operator \eqref{definition: closure1} satisfies the conditions of Proposition \ref{proposition: valuated closure axioms}.  Conditions (1), (2) and (3) are straightforward.  For Condition (4), let $x,y\in E$ and $a, b\in\mathbb{R}_{\max}$ and suppose $ax\in \mathrm{cl}(S\cup \{by\})$.  There is a finite subset $S'\subseteq S$ such that $ax\in \mathrm{cl}(S'\cup\{by\})$.  Note that $\mathrm{cl}(\{by\})$ is a weak atom of $Q_M$ (it maps to the weak atom $by$ under the isomorphism $\psi$), so $\mathrm{cl}(\emptyset)\preceq \mathrm{cl}(\{by\})$ and by semimodularity of $Q_M$, we have a weak cover
\begin{equation}
\mathrm{cl}(S')\preceq \mathrm{cl}(S'\cup \{by\}).    
\end{equation}
Combining this observation with $\mathrm{cl}(S')\subseteq\mathrm{cl}(S'\cup\{ax\})\subseteq\mathrm{cl}(S'\cup \{by\})$, we obtain some $c\leq b$ such that 
\begin{equation}
\mathrm{cl}(S'\cup\{ax\}) = \mathrm{cl}(S'\cup\{cy\}).    
\end{equation}
This implies  $\mathrm{cl}(S\cup\{ax\}) = \mathrm{cl}(S\cup\{cy\})$.  Then $ax\in \mathrm{cl}(S\cup \{cy\})$ and $cy\in\mathrm{cl}(S\cup\{ax\})$.  

For Condition (5), we suppose $ax\in \mathrm{cl}(S\cup\{bx\})$ and $b < a$.  Then $ax\in\mathrm{cl}(S'\cup\{bx\})$ for some finite $S'\subseteq S$.  Then $\mathrm{cl}(S' \cup\{bx\}) = \mathrm{cl}(S'\cup \{ax\})$.  By the strongly torsion-free condition on $Q_M$, either $a = b$ (a contradiction) or $\mathrm{cl}(S') = \mathrm{cl}(S'\cup\{ax\})$ which implies $ax\in\mathrm{cl}(S')\subseteq\mathrm{cl}(S)$.

By Proposition \ref{proposition: valuated matroid closure cryptomorphism}, there is a unique valuated matroid whose closure operator is $\mathrm{cl}$.  Also $\mathrm{cl}(\emptyset) = \{\textbf{0}\}$ and $\mathrm{cl}(ax) = \{bx\mid b\leq a\}$ for $ax\in \hat{M}$ (since $ax$ is a weak atom).  Thus the valuated matroid is simple (a circuit set with one element would give rise to an element of $\mathrm{cl}(\emptyset)$, while a circuit set with two elements would contradict our characterization of $\mathrm{cl}(ax)$).  
\end{proof}

\begin{mythm}\label{theorem: geometric modules2}
There exists a one-to-one correspondence between projective equivalence classes of simple valuated matroids and isomorphism classes of finitely generated geometric $\mathbb{R}_{\max}$-modules. 
\end{mythm}
\begin{proof}
Let $X$ be the set of projective equivalence classes of simple valuated matroids and $Y$ be the set of isomorphism classes of $\mathbb{R}_{\max}$-modules. For each $[M] \in X$, we define
\[
\Phi([M])=[Q_M],
\]
where $[M]$ is the projective equivalence class of a simple valuated matroid $M$ and $[Q_M]$ is the isomorphism class of $Q_M$.

We first check that $\Phi$ is well-defined. Suppose that $M' \in [M]$, i.e., there exists a bijection $\psi:\hat{M} \to \hat{M}'$ preserving closure. Since $\psi$ preserve closure, it induces a join-preserving bijection $\widetilde{\psi}$ between the set $Q_M$ of finitely generated flats of $M$ and $Q_{M'}$ of finitely generated flats of $M'$. Moreover, since $\psi$ is compatible with the $\mathbb{R_{\max}}$ action, $\widetilde{\psi}$ preserves scalar multiplication. and hence $\widetilde{\psi}$ is an isomorphism between $Q_M$ and $Q_{M'}$. 

Surjectivity of $\Phi$ directly follows from Theorem \ref{theorem: geometric modules1}. For injectivity, suppose that $\Phi([M])=[N]=\Phi([M'])$. Let $A_N$ be the set of weak atoms of $N$. From the construction of Theorem \ref{theorem: geometric modules1}, we have that there is a bijection 
\[
\psi:\hat{M} \to A_N
\]
which sends the closure operator of $M$ to the one given by
\[
\text{cl}(S)=\{y \in A_{N} \mid y \leq \sum x_i \text{ for some } x_i \in S\}.
\]
Moreover, $\psi$ preserves the $\mathbb{R}_{\max}$ action. Likewise, there is a bijection $\psi':\hat{M}' \to A_N$ satisfying the same properties. Consider
\[
\varphi:=(\psi')^{-1}\circ \psi:\hat{M} \to \hat{M}'.
\]
Then, one can easily check that $\varphi$ is a bijection which preserves the $\mathbb{R}_{\max}$ action and closure. In particular, $[M]=[M']$, showing that $\Phi$ is injective. 
\end{proof}

Note that in the above theorems, the finiteness of $E$ is only explicitly used to show the corresponding geometric $\mathbb{R}_{\max}$-module is finitely generated. The following example shows that the coordinate semiring of an affine tropical variety is naturally equipped with a geometric $\mathbb{R}_{\max}$-module structure associated to an infinite valuated matroid, defined in Appendix \ref{section: infinite val mat}.

\begin{myeg}\label{example: tropical var is val}
Let $K$ be a valued field with a surjective valuation $v: K \rightarrow \mathbb{R}_{\max}$. We let $A = K[x_1, \ldots, x_n] / I$ be a finitely generated $K$-algebra (equipped with a choice of generating set).  Let $E\subseteq A$ be the set of primitive monomials, i.e. elements of the form $x_1^{k_1}\ldots x_n^{k_n}$.  Because $A$ is a vector space over $K$, we obtain an infinite linear valuated matroid $M$ on $E$ as in Definition \ref{definition: infinitem val mat}.

The coordinate semiring of the tropicalization of $\mathrm{Spec} \,A$ is defined as follows:
\begin{equation}\label{eq: coordinate}
\mathbb{R}_{\max}[x_1, \ldots, x_n] / \mathrm{Bend}(I),
\end{equation}
where $\mathrm{Bend}(I)$ is the congruence relation generated by the bend relations associated to $I$; see Section \ref{subsection: semirings} or \cite{giansiracusa2016equations}. One may observe that the coordinate semiring \eqref{eq: coordinate} is precisely $Q_M$ for the valuated matroid described above. To see this, note that the coordinate semiring is the semiring of the $\mathcal{O}_K$-submodules of $A$ which are generated by finitely many monomials (c.f. \cite[Corollary 6.11]{JMT20}).  Then one can apply Proposition \ref{proposition: flats of linear valuated matroids} (2) to identify $\mathcal{O}_K$-submodules of $A$ which are generated by monomials with flats.  By identifying compact elements of the lattices of flats and of monomial submodules, we see that the coordinate semiring consists of finitely generated flats.

As a consequence, the coordinate semiring of a tropical variety is a geometric $\mathbb{R}_{\max}$-module.
\end{myeg}

The following lemma shows the theory of geometric $K$-modules is a special case of geometric $\mathcal{O}_K$-modules.

\begin{lem}
Let $K$ be an idempotent semifield, and $N$ be a $K$-module.  Then $N$ is geometric (resp. weakly atomic, atomic, semimodular, torsion-free, or strongly torsion-free) as a $K$-module if and only if it is geometric (resp. weakly atomic, atomic, semimodular, torsion-free, or strongly torsion-free) as an $\mathcal{O}_K$-module.
\end{lem}
\begin{proof}
We have already seen the cases of torsion-free or strongly torsion-free in Lemma \ref{lemma: torision-free lemma}.
Observe that the maps $\phi_{x, y}$ for $x,y\in N$ do not depend on whether $N$ is viewed as a $K$-module or $\mathcal{O}_K$-module.  In addition, the notion of weak cover depends only on the maps $\phi_{x,y}$, and the notions of semimodular modules or weakly atomic modules depend only on the addition operation and the notion of weak cover.  Similarly, the notion of atom only depends on the maps $\phi_{0,x}$, establishing the claim for atomic modules.
\end{proof}

If $x$ is an element of a geometric lattice, then $\{y\mid y\leq x\}$ is also a geometric lattice.  This fails for geometric $K$-modules since these sets are not closed under scalar multiplication.  However, this does work in the $\mathcal{O}_K$-case, and in the case of geometric $K$-modules, this set is a geometric $\mathcal{O}_K$-module.

\begin{pro}
Let $K$ be an idempotent semifield and let $N$ be a geometric $\mathcal{O}_K$-module.  Let $L\subseteq N$ be a subtractive submodule.  Then $L$ is geometric. In particular, if $x\in N$ then $\{y\in N\mid y\leq x\}$ is geometric.
\end{pro}
\begin{proof}
Let $x, y\in L$ with $x\leq y$.  Since $L$ contains all elements bounded by $y$, we obtain the same map $\phi_{x,y}$ and the same interval $[x, y]$ whether we view $x, y$ as elements of $N$ or $L$.  The claim that $L$ is atomic and semimodular follows easily.

To show $L$ is strongly torsion-free, suppose $x + ay = x + by$ in $L$ with $a,b\leq 1$.  Then this occurs in $N$ so $a = b$ or $x = x + ay = x + by$.  The other condition of Definition \ref{definition: strongly torsionfree} is equally easy.
\end{proof}

\appendix 

\section{Infinite valuated matroids} \label{section: infinite val mat}
We will define an infinite valuated matroid as a filtered union of finite valuated matroids.  Note that we do not require the underlying set to actually be infinite; ``potentially infinite valuated matroid'' would be a more accurate but more cumbersome term.

\begin{mydef}\label{definition: infinitem val mat}
Let $E$ be a set.  An \emph{infinite valuated matroid} on $E$ consists of a choice of valuated matroid $M|_S$ on each finite subset $S\subseteq E$ such that for any $S'\subseteq S\subseteq E$ we have $(M|_S)|_{S'} = M|_{S'}$.
\end{mydef}

We note that the underlying matroid of an infinite valuated matroid is a finitary matroid, introduced in \cite{eppolito2022infinite}. Finitary matroids form a special case of infinite matroids in \cite{bruhn2013axioms}. We note further that our definition is more general than that of \cite{dress1992valuated}.  For instance the following example shows we allow infinite valuated matroids to have infinite rank.

\begin{myeg}
Let $V$ be a vector space over a valued field $K$ and let $E\subseteq V$.  For each finite $S\subseteq E$, we have $S\subseteq V$ so we obtain a linear valuated matroid $M|_S$ on $S$.  If $S'\subseteq S\subseteq E$, then $(M|_S)|_{S'} = M|_{S'}$ is the linear valuated matroid on $S'$.
\end{myeg}

We define the closure operator for infinite valuated matroids in terms of finite valuated matroids. Let $E$ be a set. For $S\subseteq E$, we will denote
\[
\hat{M}|_S:=\{\mathbf{0}\} \cup (\mathbb{R}_{\max}^\times \times S) \subseteq \hat{M}.
\]

\begin{mydef}
Let $M$ be an infinite valuated matroid on a set $E$. If $S$ is finite, we denote the closure operator on $M|_S$ by $\mathrm{cl}_S$. The \emph{closure operator} of $M$ is defined as
\[
\mathrm{cl}(T) = \bigcup_{S\subseteq E \textrm{ finite}} \mathrm{cl}_S(T \cap \hat{M}|_S).
\]
\end{mydef}

The following lemma shows that closure operators of valuated matroids behave nicely under restriction.

\begin{lem}\label{lemma: restriction of closure operator}
\begin{enumerate}
    \item Let $M$ be a valuated matroid on a finite set $E$ and let $S\subseteq E$.  Let $x\in \hat{M}|_S$ and let $T\subseteq \hat{M}|_S$.  Then $x\in \mathrm{cl}(T)$ if and only if $x\in \mathrm{cl}_S(T)$.
    \item The same holds if $M$ is an infinite valuated matroid on an arbitrary set $E$.  Consequently, $\mathrm{cl}$ determines $\mathrm{cl}_S$ as $\mathrm{cl}_S(T) = \mathrm{cl}(T) \cap \hat{M}|_S$.
\end{enumerate}
\end{lem}
\begin{proof}
    (1) Let $x\in \mathrm{cl}(T)$.  Then either $x\in T_\downarrow\subseteq \mathrm{cl}_S(T)$ or there is a circuit set $C$ with $x\in C$ and $C\backslash\{x\}\subseteq T_\downarrow$.  Since $C\subseteq\hat{M}|_S$, $C$ is a circuit set of $M_S$, and hence $x\in \mathrm{cl}_S(T)$.  The converse is similar.

    (2) $\mathrm{cl}_S(T) = \mathrm{cl}_S(T\cap \hat{M}|_S)\subseteq\mathrm{cl}(T)$ by definition, so suppose $x\in \mathrm{cl}(T) \cap \hat{M}|_S$. There is a finite subset $S'\subseteq E$ such that $x\in \mathrm{cl}_{S'}(T) \cap \hat{M}|_S$.  By adding one element if necessary, we may assume $x\in \hat{M}|_{S'}$ so $x\in \mathrm{cl}_{S'}(T) \cap \hat{M}|_{S\cap S'}$.  By part (1) (applied to $S\cap S'\subseteq S')$, $x\in \mathrm{cl}_{S\cap S'}(T)$.  By part (1) (now applied to $S\cap S'\subseteq S$), we get $x\in \mathrm{cl}_S(T)$ as desired.
\end{proof}
\begin{lem}
    The closure operator of an infinite valuated matroid $M$ on a set $E$ is an algebraic closure operator.
\end{lem}
\begin{proof}Clearly if $T\subseteq T'$ then $\mathrm{cl}(T)\subseteq \mathrm{cl}(T')$.  Also, we have
\[
T = \bigcup_{S\subseteq E \textrm{ finite}} T \cap \hat{M}|_S \subseteq \mathrm{cl}(T).
\]
To see $\mathrm{cl}$ is finitary, suppose $x\in \mathrm{cl}(T)$.  Then there is some finite $S\subseteq E$ such that
\[
x\in \mathrm{cl}_S(T \cap \hat{M}|_S).
\]
Since $\mathrm{cl}_S$ is finitary, there is some finite $T'\subseteq T$ such that $x\in \mathrm{cl}_S(T' \cap \hat{M}|_S) \subseteq \mathrm{cl}(T')$.

Let $x\in \mathrm{cl}(\mathrm{cl}(T))$.  There is a finite subset $\{y_1, \ldots, y_n\}\subseteq \mathrm{cl}(T)$ such that $x\in \mathrm{cl}(\{y_1, \ldots, y_n\})$.  There is a finite subset $S'\subseteq E$ such that $x\in \mathrm{cl}_{S'}(\{y_1, \ldots, y_n\})$.  There exist finite subsets $S_i\subseteq E$ such that $y_i \in \mathrm{cl}_{S_i}(T \cap \hat{M}|_{S_i})$.  By Lemma \ref{lemma: restriction of closure operator}, these properties still hold if $S_i$ and $S'$ are replaced with any larger finite set, for instance with $S = S' \cup \bigcup S_i$.  Hence 
\begin{equation}
    x\in \mathrm{cl}_S(\{y_1, \ldots, y_n\}) \subseteq \mathrm{cl}_S(\mathrm{cl}_S(T \cap \hat{M}|_S)) = \mathrm{cl}_S(T\cap \hat{M}|_S) \subseteq\mathrm{cl}(T).
\end{equation}
\end{proof}

\begin{lem}Let $E$ be an arbitrary set and $\hat{M} = \mathbb{R}_{\max} \times E / \{-\infty\}\times E$.  Let $\mathrm{cl}$ be an algebraic closure operator on $\hat{M}$.  For $S\subseteq E$ and define $\hat{M}|_S$ in the obvious way and define $\mathrm{cl}_S(T) = \mathrm{cl}(T) \cap \hat{M}|_S$.  Then $\mathrm{cl}$ satisfies the axioms listed in the conclusion of Proposition \ref{proposition: valuated closure axioms} if and only if $\mathrm{cl}_S$ satsifies these axioms for all finite $S\subseteq E$.
\end{lem}
\begin{proof}
    To illustrate the technique, we consider axiom (4) of Proposition \ref{proposition: valuated closure axioms}, leaving the rest to the reader.  Suppose first that $\mathrm{cl}$ satisfies axiom (4), i.e. for any $T\subseteq \hat{M}$, any $a,b\in \mathbb{R}_{\max}$ and any $x, y\in E$, if $ax\in \mathrm{cl}(T\cup\{by\})$ then there exists $c\leq b$ such that $ax\in \mathrm{cl}(T \cup \{cy\})$ and $cy\in \mathrm{cl}(T\cup \{ax\})$.  
    
    Assume $T\subseteq\hat{M}|_S$ and $x,y\in S$ for some finite subset $S\subseteq E$ and that $ax\in \mathrm{cl}_S(T\cup\{by\})$.  Then $ax\in \mathrm{cl}(T\cup \{by\})$.  By axiom (4), we get $c\leq b$ such that $ax\in \mathrm{cl}(T \cup \{cy\})$ and $cy\in \mathrm{cl}(T\cup \{ax\})$.  Since $ax\in \hat{M}|_S$, $ax\in \mathrm{cl}(T\cup\{cy\})\cap \hat{M}|_S = \mathrm{cl}_S(T\cup\{cy\})$ and similarly $cy\in \mathrm{cl}_S(T\cup\{ax\})$.

    Conversely, suppose that $\mathrm{cl}_S$ satisfies axiom (4) for all finite $S\subseteq E$.  Let $T\subseteq \hat{M}$, $a,b\in\mathbb{R}_{\max}$ and $x, y\in E$ be such that $ax\in \mathrm{cl}(T\cup\{by\})$.  Due to Remark \ref{remark: checking axioms on finite sets}, we will assume $T$ is finite and hence that $T\subseteq\hat{M}|_S$ for some finite $S\subseteq E$. Enlarging $S$ if necessary, assume $x, y\in S$.  Now $ax\in \mathrm{cl}(T\cup\{by\}) \cap \hat{M}|_S = \mathrm{cl}_S(T\cup \{by\})$.  Applying axiom (4) to $\mathrm{cl}_S$, we get $c\leq b$ such that $ax\in \mathrm{cl}_S(T \cup \{cy\})\subseteq\mathrm{cl}(T \cup \{cy\})$ and $cy\in \mathrm{cl}_S(T\cup \{ax\})\subseteq\mathrm{cl}(T\cup \{ax\})$. 
\end{proof}

\begin{mythm}Let $E$ be an arbitrary set and $\hat{M} = \mathbb{R}_{\max} \times E / \{-\infty\}\times E$.  There is a one-to-one correspondence between infinite valuated matroids on $E$ and algebraic closure operators on $\hat{M}$ satisfying the conclusions of Proposition \ref{proposition: valuated closure axioms}.
\end{mythm}
\begin{proof}
    Given an infinite valuated matroid $M$ on $E$, we get a closure operator $\mathrm{cl}$ such that the closure operators $\mathrm{cl}_S(T) = \mathrm{cl}(T)\cap \hat{M}|_S$ are those associated to $M|_S$ (by Lemma \ref{lemma: restriction of closure operator}).  In particular these closure operators satisfy the conclusions of Proposition \ref{proposition: valuated closure axioms}.  Hence so does $\mathrm{cl}$.

    We have constructed a map from valuated matroids to algebraic closure operators satisfying the desired conditions. For injectivity of this map, observe that $\mathrm{cl}$ determines $\mathrm{cl}_S$ which in turn determines $M_S$, finally determining $M$. It remains to show surjectivity.  Let $\mathrm{cl}$ be an algebraic closure operator on $\hat{M}$ satisfying the conclusions of Proposition \ref{proposition: valuated closure axioms}.  We wish to show $\mathrm{cl}$ is the closure operator of an infinite valuated matroid.

    For finite $S\subseteq E$, let $\mathrm{cl}_S$ be given by $\mathrm{cl}_S(T) = \mathrm{cl}(T) \cap \hat{M}|_S$, and observe this is the closure operator of a valuated matroid, which we denote $M|_S$.  For $S'\subseteq S$ and $T\subseteq \hat{M}|_{S'}$, $\mathrm{cl}_{S'}(T) = \mathrm{cl}(T) \cap \hat{M}|_{S'} = \mathrm{cl}_S(T) \cap \hat{M}|_{S'}$.  By Lemma \ref{lemma: restriction of closure operator}, this implies $\mathrm{cl}_{S'}$ is the closure operator of $(M|_S)_{S'}$ and hence $(M|_S)_{S'} = M|_{S'}$.  Thus the family $M|_S$ yields a valuated matroid $M$.  Let $\mathrm{cl}'$ be the closure operator of $M$.  
    
    For finite $T\subseteq\hat{M}$, there is some $S'\subseteq E$ such that $T\subseteq \hat{M}|_{S'}$.  This then holds for all finite $S$ containing $S'$.  Then for such $S$, $\mathrm{cl}(T) \cap \hat{M}|_S = \mathrm{cl}_S(T) = \mathrm{cl}'(T) \cap \hat{M}|_S$.  Since $\hat{M} = \bigcup_{S\subseteq T \textrm{ finite; } S'\subseteq S} \hat{M}_S$, we obtain $\mathrm{cl} = \mathrm{cl}'$ on finite subsets.  Since both closure operators are algebraic, $\mathrm{cl} = \mathrm{cl}'$ on arbitrary subsets.  Hence $\mathrm{cl}$ is the closure operator of an infinite valuated matroid.
\end{proof}

\begin{mythm}There exists a one-to-one correspondence between projective equivalence classes of simple infinite valuated matroids and isomorphism classes of geometric $\mathbb{R}_{\max}$-modules.
\end{mythm}
\begin{proof}
The proof is the same as Theorems \ref{theorem: geometric modules1} and \ref{theorem: geometric modules2} and Proposition \ref{proposition: $Q_M$ geometric} minus the steps where we show finite generation is equivalent to finiteness of $E$.
\end{proof}

\bibliography{closure}\bibliographystyle{alpha}

\end{document}